\documentclass[11pt]{amsart}
\usepackage{amsmath,amssymb,epsfig,color, amsthm, mathrsfs}
\usepackage{wrapfig,lipsum,floatrow,sidecap,comment}
\usepackage{supertabular}
\usepackage{graphicx}
\usepackage{multirow}
\usepackage{hhline}
\usepackage{url}
\usepackage{mathtools}
\graphicspath{ {images/} }
\allowdisplaybreaks

\graphicspath{ {images/} }

\def\b{\textbf{b}}
\def\L{\mathcal{L}}

\newtheorem{theorem}{Theorem}[section]
\newtheorem{conjecture}[theorem]{Conjecture}
\newtheorem{corollary}[theorem]{Corollary}
\newtheorem{definition}[theorem]{Definition}

\newtheorem{proposition}[theorem]{Proposition}
\newtheorem{remark}[theorem]{Remark}

\newtheorem{examples}[theorem]{Examples}

\begin{document}
\title{The Braid Indices of Pretzel Links: \\
A Comprehensive Study, Part II}
\author{Yuanan Diao$^\dagger$, Claus Ernst$^*$ and Gabor Hetyei$^\dagger$}
\address{$^\dagger$ Department of Mathematics and Statistics\\
University of North Carolina Charlotte\\
Charlotte, NC 28223}
\address{$^*$ Department of Mathematics\\
Western Kentucky University\\
Bowling Green, KY 42101, USA}
\email{}
\subjclass[2010]{Primary: 5725; Secondary: 5727}
\keywords{knots, links, braid index, alternating links, Seifert graph, Morton-Frank-Williams inequality, pretzel link, Montesinos link.}

\begin{abstract}
This paper is the second part of our comprehensive study on the braid index problem of pretzel links. Our ultimate goal is to completely determine the braid indices of all pretzel links, alternating or non alternating. In our approach, we divide the pretzel links into three types as follows. Let $D$ be a standard diagram of an oriented pretzel link $\L$, $S(D)$ be the Seifert circle decomposition of $D$, and $C_1$, $C_2$ be the Seifert circles in $S(D)$ containing the top and bottom long strands of $D$ respectively, then $\L$ is classified as a Type 1 (Type 2) pretzel link if $C_1\not=C_2$  and $C_1$, $C_2$ have different (identical) orientations. In the case that $C_1=C_2$, then $\L$ is classified as a Type 3 pretzel link. In our previous paper, we succeeded in reaching our goal for all Type 1 and Type 2 pretzel links. That is, we successfully derived precise braid index formulas for all Type 1 and Type 2 pretzel links. In this paper, we present the results of our study on Type 3 pretzel links. In this case, we are very close to reaching our goal. More precisely, with the exception of a small percentage of Type 3 pretzel links, we are able to determine the precise braid indices for the majority of Type 3 pretzel links. Even for those exceptional ones, we are able to determine their braid indices within two consecutive integers. With some numerical evidence, we conjecture that in such a case, the braid index of the Type 3 pretzel link is given by the larger of the two consecutive integers given by our formulas.
\end{abstract}

\maketitle
\section{Introduction}\label{sec:intro}

This paper is the second part of our comprehensive study on braid indices of non alternating pretzel links. Our study is motivated by the lack of general knowledge on braid indices of non alternating links while much more is known in the case of alternating links. For example, all alternating Montesinos links have known braid indices~\cite{Diao2021} (which include all rational links and alternating pretzel links), yet very little is known about the braid indices of non alternating Montesinos links, including the non alternating pretzel links. By studying the braid indices of the non alternating pretzel links, which contain surprisingly many cases that are quite different from their alternating counterparts, we hope our work will shed some light for future studies and add a rich set of non alternating links with known braid indices in the literature. 

\medskip
Let $D$ be a standard diagram of an oriented pretzel link $\L$ and $S(D)$ be the Seifert circle decomposition of $D$. Let $C_1$ and $C_2$ be the Seifert circles in $S(D)$ containing the top and bottom strands of $D$ respectively. Then $D$ is said to be a Type 1 (Type 2) pretzel link if $C_1\not=C_2$  and $C_1$, $C_2$ have different (identical) orientations. In the case that $C_1=C_2$, then $D$ is said to be a Type 3 pretzel link. In our previous paper~\cite{Diao2024}, we have succeeded in determining the braid indices for all Type 1 and Type 2 pretzel links. This paper continues and completes the work in~\cite{Diao2024} by finding the braid indices of all Type 3 pretzel links, with only one small exceptional class of Type 3 non alternating pretzel links, where we can only determine their braid indices to be within two consecutive integers. Similar to the approach used in our first paper~\cite{Diao2024}, we determine the braid index of a link by determining its lower bound using the Morton-Frank-Williams inequality (MFW inequality)~\cite{FW,Mo}, and its upper bound by direct construction. The MFW inequality states that the braid index $\b(\L)$ of any link $\L$ is bounded below by $\b_0(\L)=(E(\L)-e(\L))/2+1$, where $E(\L)-e(\L)$ is the degree span of the variable $a$ in the HOMFLY-PT polynomial $H(\L,z,a)$ of $\L$ with $E(\L)$ and $e(\L)$ being the highest and lowest powers of the variable $a$ in $H(\L,z,a)$. Thus, in order to apply the MFW inequality, one has to be able to determine $E(\L)$ and $e(\L)$ and this is the first hurdle one has to overcome. On the other hand, a well known result due  to Yamada~\cite{Ya}, states that the number of Seifert circles in any diagram $D$ of $\L$ is an upper bound for $\b(\L)$. So, if one can present $\L$ by a diagram $D$ whose number of Seifert circles, denoted by $s(D)$, is the same as $\b_0(\L)$, then one succeeds with $\b(\L)=s(D)=\b_0(\L)$. Otherwise, one obtains an inequality $\b_0(\L)\le \b(\L)\le s(D)$. Thus, the second hurdle one has to overcome is to construct a diagram of $\L$ with the desired number of Seifert circles, namely $\b_0(\L)$. However, this method has a caveat since it is known that the MFW inequality is not sharp for some links. Thus when one fails to find a diagram of $\L$ with $\b_0(\L)$ Seifert circles, one cannot be sure whether it is because the braid index of $\L$ is really larger than $\b_0(\L)$, or one has not found the optimal construction yet. This is precisely the situation we run into for the exceptional class of Type 3 pretzel links, see (\ref{MT1e4}) in Theorem \ref{MT1}.

\medskip
Figure \ref{Mont} shows a few examples of Type 3 pretzel links. A non alternating Type 3 pretzel link may not be a minimum diagram. In fact the pretzel link in the middle of Figure \ref{Mont} is the unknot. We note that alternating Type 3 pretzel links belong to the set of Type  B Montesinos links defined in~\cite{Diao2021}. If $D$ is a standard diagram of a Type 3 pretzel link, then the Seifert circle decomposition of $D$ contains a cycle of Seifert circles (the adjacent Seifert circles in the cycle share crossings) as shown in Figure \ref{Mont}. We shall call this cycle of Seifert circles the {\em main cycle} of Seifert circles of $D$. Without loss of generality, we can always orient the top long strand in a pretzel link diagram from right to left. This choice, together with the crossing sign and the information on how the crossings in a strip are to be smoothed (either all vertically or all horizontally), will allow us to determine the diagram completely. Thus in our notations introduced below, we only need to indicate the sign of the crossings in each strip. 

\medskip
\begin{definition}\label{TypeDef}{\em 
We introduce the following {\em grouping} of Type 3 pretzel links. Consider a Type 3 pretzel link $\L$, which has a diagram with the structure of a Montesinos link diagram. We divide its strips of crossings into two parts. The first part contains the strips whose strings have parallel orientations and the second part contains the strips whose strings have antiparallel orientations. We then use the notation
$P_3(\mu_1,\ldots,\mu_{\rho^+};-\nu_1,\ldots,-\nu_{\rho^-}\vert 2\alpha_1,\ldots, 2\alpha_{\kappa^+}; -2\beta_1,\ldots, -2\beta_{\kappa^-})$ to denote the set of all Type 3 pretzel links with strips containing $\mu_1,\ldots,\mu_{\rho^+}$ positive crossings and $\nu_1,\ldots,\nu_{\rho^-}$ negative crossings in the first part, and strips with $2\alpha_1,\ldots, 2\alpha_{\kappa^+}$ positive crossings and $2\beta_1,\ldots, 2\beta_{\kappa^-}$ negative crossings in the second part.
}
\end{definition}

\medskip
\begin{figure}[htb!]
\includegraphics[scale=.7]{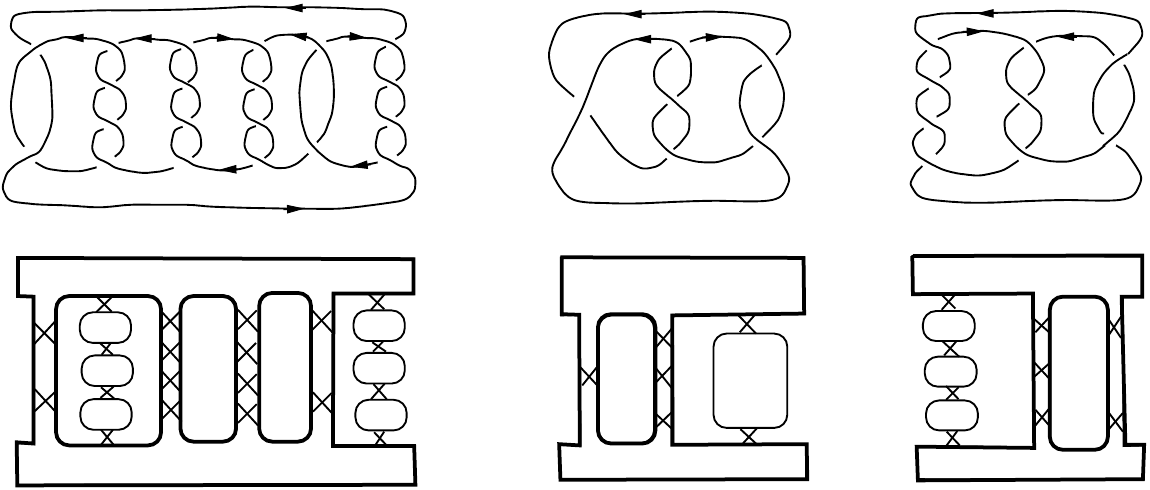}
\caption{Three examples of Type 3 pretzel links. Left: $D\in P_3(2;-4,-4,-2\vert 4,4;0)$; Middle: $D\in P_3(1;-3\vert 2;0)$; Right: $D\in P_3(2;-3\vert 4;0)$. The main cycles of Seifert circles in the examples are highlighted by thick lines.
}
\label{Mont} 
\end{figure}

\begin{remark}\label{remark1.1}{\em 
In the case of a Type 3 pretzel link, we must have $\kappa^++\kappa^-=2n$ for some $n\ge 1$, due to the fact that two non-concentric Seifert circles that share crossings must have opposite orientations. For the same reason, a strip with an odd number of half twists must belong to the first part in the grouping of a Type 3 pretzel link.
 }
 \end{remark}

\medskip
\begin{remark}\label{remark1.2}{\em 
For the sake of simplicity, we do not allow one crossing strips with different crossing signs in a pretzel link diagram since such crossings can be pairwise deleted via Reidemeister II moves. That is, we shall assume that in Definition \ref{TypeDef} that either $\mu_j>1$ for all $j$ or $\nu_i>1$ for all $i$. A pretzel link diagram with its top strand oriented from right to left and satisfying the condition that crossings in strips containing only one crossing have the same crossing sign is called a {\em standard} pretzel link diagram in this paper. Since the braid indices of a link and its mirror image are the same, we only need to consider Type 3 pretzel links with $\nu_i>1$ for all $i$ (in the case that $\rho^->0$). We summarize our results in the following theorems. }
\end{remark}

\medskip
In the statements of the theorems, we shall assume, WLOG, that $\mu_1\ge \mu_2 \ldots \ge \mu_{\rho^+}$, $\nu_1\ge \nu_2 \ldots \ge \nu_{\rho^-}$,  $\alpha_1\ge \alpha_2 \ldots \ge \alpha_{\kappa^+}$ and $\beta_1\ge \beta_2 \ldots \ge \beta_{\kappa^-}$. This does not mean that the corresponding strips are ordered like this in the diagram. We shall denote by $\delta^+$ the number of $\mu_j$'s such that $\mu_j=1$ by $\delta^-$ the number of $\nu_i$'s such that $\nu_i=1$. Note that $\delta^+\cdot \delta^-=0$ by Remark \ref{remark1.2}.

\begin{theorem}\label{MT1}
Let $\L\in P_3(\mu_1,\ldots,\mu_{\rho^+};-\nu_1,\ldots,-\nu_{\rho^-}\vert 2\alpha_1,\ldots, 2\alpha_{\kappa^+}; -2\beta_1,\ldots, -2\beta_{\kappa^-})$ such that $\rho^++\rho^-=2$, then
\begin{eqnarray}
\b(\L)&=&\left\{
\begin{array}{ll}
1,\ &{\rm if}\ \rho^+=\rho^-=1, \kappa^+=\kappa^-=0, |\mu_1-\nu_1|=1,\\
2,\ &{\rm if}\ \rho^+=\rho^-=1,  \kappa^+=\kappa^-=0, |\mu_1-\nu_1|\not=1,
\end{array}
\right.
 \label{MT1e0}
\end{eqnarray}
\begin{eqnarray}
\b(\L)&=&\left\{
\begin{array}{ll}
\sum \alpha_j,\ &{\rm if}\ \left\{
\begin{array}{l}
 \kappa^-=0, \kappa^+=1, \mu_1=1, \nu_1=3, \alpha_{1}=1,\ {\rm or}\\
  \kappa^-=0, \kappa^+>1, \mu_1=1, \nu_1=3, \alpha_{\kappa^+}=1, \alpha_{\kappa^+-1}>2, \ {\rm or}\\
 \kappa^-=0, \kappa^+> 0, \mu_1=1, \nu_1=2, \alpha_{\kappa^+}>1,
\end{array}
\right.\\
\sum \beta_i,\ &{\rm if}\ \left\{
\begin{array}{l}
\kappa^+=0, \kappa^-=1, \nu_1=1, \mu_1=3, \beta_{1}=1,\ {\rm or}\\
\kappa^+=0, \kappa^->1, \nu_1=1, \mu_1=3, \beta_{\kappa^-}=1, \beta_{\kappa^--1}>2, \ {\rm or}\\
\kappa^+=0, \kappa^-> 0, \nu_1=1, \mu_1=2, \beta_{\kappa^-}>1,
\end{array}
\right.
\end{array}
\right.
 \label{MT1e1}
\end{eqnarray}
\begin{eqnarray}
\b(\L)&=&\left\{
\begin{array}{ll}
1+\sum \alpha_j,\ &{\rm if}\ \left\{
\begin{array}{l}
\kappa^-=0,  \kappa^+>0, \mu_1>1, \nu_1=\mu_1+1, \alpha_{\kappa^+}>\mu_1,\ {\rm or}\\
\kappa^-=0,  \kappa^+>1, \mu_1=1, \nu_1=3, \alpha_{\kappa^+}=\alpha_{\kappa^+-1}=1,\ {\rm or}\\
\kappa^-=0, \kappa^+> 0, \mu_1=1, \nu_1=2, \alpha_{\kappa^+}=1,\ {\rm or}\\
\kappa^-=0, \kappa^+> 0, \mu_1=1, \nu_1=3, \alpha_{\kappa^+}>1,\ {\rm or}\\
\kappa^-=0, \kappa^+>0, \mu_1=1, \nu_1\ge 4,
\end{array}
\right.\\
1+\sum \beta_i,\ &{\rm if}\ \left\{
\begin{array}{l}
\kappa^+=0,  \kappa^->0, \nu_1>1, \mu_1=\nu_1+1, \beta_{\kappa^-}>\nu_1,\ {\rm or}\\
\kappa^+=0,  \kappa^->1, \mu_1=3, \nu_1=1, \beta_{\kappa^-}=\beta_{\kappa^--1}=1,\ {\rm or}\\
\kappa^+=0, \kappa^-> 0, \mu_1=2, \nu_1=1, \beta_{\kappa^-}=1,\ {\rm or}\\
\kappa^+=0, \kappa^-> 0, \mu_1=3, \nu_1=1, \beta_{\kappa^-}>1,\ {\rm or}\\
\kappa^+=0, \kappa^->0, \mu_1\ge 4, \nu_1=1,
\end{array}
\right.
\end{array}
\right.
 \label{MT1e2}
\end{eqnarray}
\begin{eqnarray}
\b(\L)&=&\left\{
\begin{array}{ll}
2+\sum \alpha_j,\ &{\rm if}\ \left\{
\begin{array}{l}
\kappa^-=0, \kappa^+>0, \mu_1>1, \nu_1=\mu_1+1, \alpha_{\kappa^+}=1,\ {\rm or}\\
\kappa^-=0, \kappa^+>0, \mu_1>1, \nu_1\not=\mu_1+1,\ {\rm or}\\
\rho^+=0, \rho^-=2, \kappa^-=0, \kappa^+>0,
\end{array}
\right.\\
2+\sum \beta_i,\ &{\rm if}\ \left\{
\begin{array}{l}
\kappa^+=0, \kappa^->0, \nu_1>1, \mu_1=\nu_1+1, \beta_{\kappa^-}=1,\ {\rm or}\\
\kappa^+=0, \kappa^->0, \nu_1>1, \mu_1\not=\nu_1+1,\ {\rm or}\\
\rho^+=2, \rho^-=0, \kappa^+=0, \kappa^->0, 
\end{array}
\right.
\end{array}
\right.
 \label{MT1e3}\\
  \b(\L)&\in&\left\{
\begin{array}{ll}
  \{1+\sum \alpha_j,2+\sum \alpha_j\}\ &{\rm if}\ 
\kappa^-=0, \kappa^+>0, \mu_1>1, \nu_1=\mu_1+1, 1<\alpha_{\kappa^+}\le \mu_1, \\
 \{1+\sum \beta_i,2+\sum \beta_i\}\ &{\rm if}\ \kappa^+=0, \kappa^->0,  \nu_1>1, \mu_1=\nu_1+1, 1<\beta_{\kappa^-}\le \nu_1,\\
  \{\sum \alpha_j,1+\sum \alpha_j\}\ &{\rm if}\ 
\kappa^-=0, \kappa^+>1, \mu_1=1, \nu_1=3, \alpha_{\kappa^+}=1, \alpha_{\kappa^+-1}= 2,\\
 \{\sum \beta_i,1+\sum \beta_i\}\ &{\rm if}\ \kappa^+=0, \kappa^->1, \mu_1=3, \nu_1=1, \beta_{\kappa^-}=1, \beta_{\kappa^--1}= 2,
  \end{array}
  \right. \label{MT1e4}\\
  \b(\L)&=&\sum \alpha_j+\sum \beta_i,  \ {\rm if}\ 
  \left\{
\begin{array}{l}
\rho^+= \delta^+=2, \rho^-=0, \kappa^+\ge 2,\ {\rm or}\\
\rho^-=\delta^-=2, \rho^+=0, \kappa^-\ge 2,
\end{array}
  \right.\label{MT1e5}\\
   \b(\L)&=&1+\sum \alpha_j+\sum \beta_i, \ {\rm if}\ \left\{
\begin{array}{l}
\kappa^->0, \kappa^+>0, \mu_1=1, \nu_1\ge 2,\ {\rm or}\\
\kappa^+>0, \kappa^->0, \mu_1\ge 2, \nu_1=1,\ {\rm or}\\
\rho^+=2, \rho^-=0, \delta^+\ge 1, \kappa^+=1,\ {\rm or}\\
\rho^-=2, \rho^+=0, \delta^-\ge 1, \kappa^-=1,\ {\rm or}\\
\rho^+=2, \rho^-=0, \delta^+= 1, \kappa^+\ge 1,\ {\rm or}\\
\rho^-=2, \rho^+=0, \delta^-= 1, \kappa^-\ge 1,
\end{array}
  \right.\label{MT1e6}\\
 \b(\L)&=&2+\sum \alpha_j+\sum \beta_i, \ {\rm if}\ 
  \left\{
\begin{array}{l}
\kappa^->0,  \mu_1>1, \nu_1=\mu_1+1,\ {\rm or}\\
\kappa^+>0, \nu_1>1, \mu_1=\nu_1+1,\ {\rm or}\\
\mu_1>1,  \nu_1>1, |\mu_1-\nu_1|\not=1,\ {\rm or}\\
\rho^+=2, \rho^-=0, \kappa^+>0, \delta^+=0, \ {\rm or}\\
\rho^-=2, \rho^+=0, \kappa^->0, \delta^-=0.
\end{array}
  \right.\label{MT1e7}
\end{eqnarray}
\end{theorem}

\medskip
\begin{theorem}\label{MT2}
Let $\L\in P_3(\mu_1,\ldots,\mu_{\rho^+};-\nu_1,\ldots,-\nu_{\rho^-}\vert 2\alpha_1,\ldots, 2\alpha_{\kappa^+}; -2\beta_1,\ldots, -2\beta_{\kappa^-})$ such that $\rho^++\rho^-=2n\ge 4$, then
\begin{eqnarray}
\b(\L)&=&
\left\{
\begin{array}{ll}
2n-\kappa^+-\min\{\delta^+-\kappa^+,n-1\}+\sum \alpha_j+\sum \beta_i,\ &{\rm if}\ \delta^+> \rho^-+\kappa^+,\\
2n-\kappa^--\min\{\delta^--\kappa^-,n-1\}+\sum \alpha_j+\sum \beta_i \  &{\rm if}\ \delta^-> \rho^++\kappa^-,
\end{array}
\right.\label{MT2e1}\\
\b(\L)&=&
\left\{
\begin{array}{ll}
2n-\delta^++\sum \alpha_j+\sum \beta_i \  &{\rm if}\ \delta^-=0\ {\rm and}\  \delta^+\le \rho^-+\kappa^+,\\
2n-\delta^-+\sum \alpha_j+\sum \beta_i \  &{\rm if}\ \delta^+=0\ {\rm and}\  \delta^-\le \rho^++\kappa^-.
\end{array}
\right.\label{MT2e2}
\end{eqnarray}
\end{theorem}

\begin{remark}\label{remark1.3}{\em 
We note that in the statements of Theorems \ref{MT1} and \ref{MT2}, as well as in the rest of this paper, the summations $\sum \alpha_j$ and $\sum \beta_i$ are always taken over all possible values of $j$ and $i$, namely $1\le j\le \kappa^+$ and $1\le i\le \kappa^-$. In the case that $\kappa^+=0$ ($\kappa^-=0$), it is understood that $\sum \alpha_j=0$ ($\sum \beta_i=0$). In the case that the Type 3 pretzel link is alternating, the formulation of the braid index of the link here matches with that obtained in \cite{Diao2021}. The Type 3 pretzel links covered by formula (\ref{MT1e4}) satisfy the condition $\b_0(\L)\le \b(\L)\le \b_0(\L)$+1, and we know that for some of them (as we shall see in Example \ref{Ex1} it is true that $\b_0(\L)<\b(\L)=\b_0(\L)+1$. Finally, we note that the third  and fourth expression in (\ref{MT1e4}) follow directly from the first and second expression respectively as $P_3(1;-3\vert 2\alpha_1,\ldots, 2\alpha_{\kappa^+-1},2;0)=P_3(2;-3\vert 2\alpha_1,\ldots, 2\alpha_{\kappa^+-1};0)$ and $P_3(3;-1\vert 0;-2\beta_1,\ldots, -2\beta_{\kappa^--1},-2)=P_3(3;-2\vert 0;-2\beta_1,\ldots, -2\beta_{\kappa^--1})$.
}
\end{remark}

\medskip
\begin{examples}\label{Ex1}{\em
(i) $\L\in P_3(3,1,1,1,1,1;-5,-4\vert 0;0)$. Here $\rho^++\rho^-=6+2=2n=8$, $\delta^+=5$.  $\b(\L) =2n-\min\{5,3\}=5$ by (\ref{MT2e1}). The alternating counterpart of $\L$ is in $P_3(5,4,3,1,1,1,1,1;0\vert 0;0)$, which has the same braid index 5.

\smallskip\noindent
(ii) $\L\in P_3(2,2,2,1,1;-2,-2,-2\vert 0;0)$. Here $\rho^++\rho^-=5+3=2n=8$, $\delta^+=2$. $\b(\L) =2n-\delta^+=6$ by (\ref{MT2e2}). The alternating counterpart of $\L$ is in $P_3(2,2,2,2,2,2,1,1;0\vert 0;0)$, which has the same braid index 6 by (\ref{MT2e2}).

\smallskip\noindent
(iii) $\L=L9n15\{1\}\in P_3(2;-3|4;0)$ ($L9n15\{1\}$ is the notation in \cite{Liv}, in Rolfson's notation this is $9_{49}^2$). A diagram $D$ of $\L$ is shown at the right side of Figure \ref{Mont}. By (\ref{MT1e4}), we have $\b_0(\L)=3\le \L\le 4$. However, by direction computation we see that for the parallel double $\mathbb{D}$ of $D$ as shown in Figure \ref{double}, we have $\b_0(\mathbb{D})=7$. Since it is apparent that $\b(\mathbb{D})\le 2\b(\L)$,  we must have $\b(\L)=4$. The alternating counterpart of $\L$ is in $P_3(3,2;0\vert 0;-4)$ ($L9a29\{1\}$ is the notation in \cite{Liv}, in Rolfson's notation this is $9_{19}^2$), which has the same braid index 4.

\smallskip\noindent 
(iv) $\L=L11n204\{1\}\in P_3(2;-3|6;0)$. Here we have $\b(\L)=4=1+\sum \alpha_j$. A proof of this will be given in the last section. One wants to compare this example with (iii). The alternating counterpart of $\L$ is in $P_3(3,2;0\vert 0;-6)$ ($L11a277\{1\}$ is the notation in \cite{Liv}), which has a braid index of 5.

\smallskip\noindent
(v) $\L\in P_3(3,3,2;-3|0;-4,-4)$. Here $\rho^++\rho^-=3+1=2n=4$, $\delta^+=\delta^-=0$. $\b(\L) =8$ by (\ref{MT2e2}).
The alternating counterpart of $\L$ is in $P_3(3,3,3,2;0\vert 0;-4,-4)$, which has the same braid index 8 by (\ref{MT2e2}).

\smallskip\noindent
(vi) $\L\in P_3(3,1,1,1,1,1,1,1;0|4,4,2;0) $. Here $n=4$, $\delta^+=7$, $\kappa^+=3$ so $\delta^+-\kappa^+=4>n-1=3$. By (\ref{MT2e1}), $\b(\L)=8-3-3+\sum \alpha_j=7$ but its alternating counterpart (which is in $P_3(3,1,1,1,1,1,1,1;0|0,-4,-4,-2)$) has braid index $10$. This example can be generalized so that the difference between the braid indices of $\L$ and its alternating counterpart is equal to any given positive integer as follows. Let $q\ge 1$ be a given positive integer. Consider $\L\in P_3(\mu_1,\ldots,\mu_{\rho^+};0\vert 2\alpha_1,\ldots, 2\alpha_{\kappa^+}; 0)$ such that $\rho^+=2n\ge 2q$, $\kappa^+=q$ and $\delta^+>q+n-1$. Then by (\ref{MT2e1}), $\b(\L)=2n-\kappa^+-(n-1)+\sum \alpha_j$. The  alternating counterpart of $\L$ is in $P_3(\mu_1,\ldots,\mu_{2n};0\vert 0;-2\alpha_1,\ldots,-2\alpha_{\kappa^+})$, which has braid index $\b(\L)=2n-(n-1)+\sum \alpha_j$. Thus the difference between the braid indices of $\L$ and its alternating counterpart is $\kappa^+=q$.
}
\end{examples}

\medskip
\begin{figure}[htb!]
\includegraphics[scale=1.2]{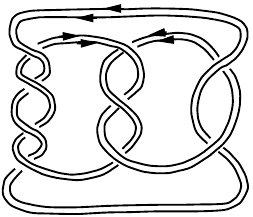}
\caption{A parallel double of $D\in P_3(2;-3\vert 4;0)$.
}
\label{double} 
\end{figure}

\medskip
As our examples above show, in some cases, the difference between the braid indices of a non alternating pretzel link and its alternating counterpart can be arbitrarily large. This means that the construction method (for determining the braid index upper bound) we use for alternating links is not enough. This is similar to the situation we run into in Part I of our study on the subject  \cite{Diao2024}, and we will also need the new construction methods introduced there for this purpose. 

\medskip
The approach we use here to prove Theorems \ref{MT1} and \ref{MT2} is similar to that used in our previous paper: we will establish the braid index lower bound by computing $\b_0(\L)=(E(\L)-e(\L))/2+1$ for each $\L$, and the prove that $\b_0(\L)$ is (in most cases) also the braid index upper bound by direction construction.
The rest of the paper is arranged as follows. In the next section, we will very briefly summarize the main tools we need for establishing the braid index lower bounds and refer the details to our previous paper. In order to compute $\b_0(\L)$, we need to determine $E(\L)$ and $e(\L)$. It turns out that $\L$ has to be divided into many different cases for this purpose. To keep these many different cases more trackable, we shall divide the derivation of $E(\L)$ and $e(\L)$ in three sections, namely Sections \ref{Type3basic_case}, \ref{Type3S2} and \ref{Type3S3}. The actual calculations of $\b_0(\L)$  are given at the end of Section \ref{Type3S3}. We will carry out the proofs for the braid index upper bounds in Section \ref{Upper_Sec} and end the paper with some remarks and open questions in Section \ref{end_sec}.

\section{preparations for establishing the lower and upper bounds}\label{bound_prep}

For the convenience of our reader, in this section we provide a list of known facts and formulas needed for our proofs and calculations. The reader can find the details in our previous paper \cite{Diao2024} and the references provided there.

\medskip
\begin{remark}\label{remark2.1}{\em We will be using the following two equivalent forms of skein relation to compute the HOMFLY-PT polynomial $H(D,z,a)$:
\begin{eqnarray}
H(D_+,z,a)&=&a^{-2}H(D_-,z,a)+a^{-1}zH(D_0,z,a),\label{Skein1}\\
H(D_-,z,a)&=&a^2 H(D_+,z,a)-azH(D_0,z,a).\label{Skein2}
\end{eqnarray}
For a link diagram $D$, we shall use $c(D)$ ($c^-(D)$) to denote the number of crossings (negative crossings) in $D$, and $w(D)=c(D)-2c^-(D)$ to denote the writhe of $D$. If we write $H(D,z,a)$ as a Laurent polynomial of $a$, then we will use  $p^h(D)$ ($p^\ell(D)$) to denote the coefficient of $a^{E(D)}$ ($a^{e(D)}$). $p^h(D)$ and $p^\ell(D)$ are Laurent polynomials of $z$ and we shall use $p^h_0(D)$ and $p^\ell_0(D)$ to denote the highest power terms of $p^h(D)$ and $p^\ell(D)$ respectively. This means that $p^h_0(D)$ and $p^\ell_0(D)$ are monomials of $z$.}
\end{remark}

\begin{figure}[htb!]
\includegraphics[scale=.6]{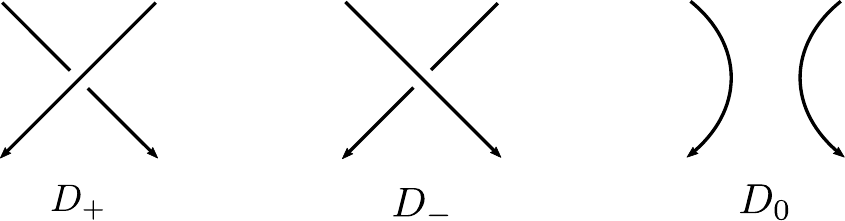}
\caption{The sign convention at a crossing of an oriented link and the splitting of the crossing: the crossing in $D_+$ ($D_-$) is positive (negative) and is assigned $+1$ ($-1$) in the calculation of the writhe of the link diagram.}
\label{fig:cross}
\end{figure}

\medskip
\begin{remark}\label{remark2.2}{\em 
The HOMFLY-PT polynomial has the following properties:

(i) $H(\L_1\#\L_2)=H(\L_1)H(\L_2)$ in general~\cite{HOMFLY}; (ii) $H(\L^r,z,a)=H(\L,z,-a^{-1})$ where $\L^r$ is the mirror image of $\L$~\cite{HOMFLY}. From this one obtains $E(\L^r)=-e(\L)$, $e(\L^r)=-E(\L)$, $p^h(\L^r)=(-1)^{e(\L)}p^\ell(\L)$ and $p^\ell(\L^r)=(-1)^{E(\L)}p^h(\L)$; (iii) $H(\L,z,a)$
does not change under a mutation move~\cite[Proposition 11]{LICKORISH}. It follows that changing the order of the strips in a pretzel link will not change the HOMFLY-PT polynomial of the link, hence all pretzel links in any one of the $P_3$ sets as defined in Definition~\ref{TypeDef} will have the same HOMFLY-PT polynomial.}
\end{remark}

\medskip
By Remark \ref{remark1.2} and Remark \ref{remark2.2}(ii), in the rest of this paper, we only need to consider Type 3 pretzel links without negative lone crossings. That is, we shall assume that $\nu_i>1$ for all $i$ in the case that $\rho^->0$. 

\medskip
\begin{remark}\label{remark2.3}{\em 
VP$^+$ and VP$^-$ refer to the procedures defined in \cite{Diao2021,Diao2024} where we apply (\ref{Skein1}) to a positive crossing or (\ref{Skein2}) to a negative crossing in a diagram $D$ of $\L$ in order to determine $E(\L)$ and $e(\L)$. }
\end{remark}

\begin{remark}\label{elementary_remark}{\em 
A crossing in a link diagram is called a {\em lone crossing} if it is the only crossing between two Seifert circles of the diagram. 
If $D$ is an alternating link diagram without any lone crossings, then by \cite[proposition 1.1]{DHL2017} we have
$$
E(D)= s(D)-w(D)-1, e(D)= -s(D)-w(D)+1
$$
with
\begin{eqnarray*}
p_0^h(D)&=&(-1)^{c^-(D)} z^{c(D)-2\sigma^-(D)-s(D)+1}, \\
p_0^\ell(D)&=& (-1)^{c^-(D)+s(D)-1} z^{c(D)-2\sigma^+(D)-s(D)+1},
\end{eqnarray*}  
where $c(D)$ ($c^-(D)$) stands for the number of crossings (negative crossings) in $D$, $w(D)=c(D)-2c^{-}(D)$ is the writhe of $D$, $\sigma^+(D)$ ($\sigma^-(D)$) is the number of pairs of Seifert circles in $S(D)$ that share multiple positive (negative) crossings. }
\end{remark}

\medskip
\begin{remark} \label{MP_move}{\em
If a diagram has a lone crossing, then either the over strand or the under strand at that crossing can be re-routed to create a new diagram with one less Seifert circle. This move is known as the {\em Murasugi-Przytycki move} (MP move for short)~\cite{Diao2021}. If the diagram contains a sequence of $k\ge 2$ Seifert circles connected by $k-1$ lone crossings, then at least $\lfloor \frac{k}{2}\rfloor$ MP moves can be performed as shown in Figure \ref{MP_move}. We denote by $r^+(D)$ ($r^-(D)$) the number of MP moves that can be made on positive (negative) lone crossings in $D$.
}
\end{remark}

\begin{figure}[!htb]
\includegraphics[scale=0.5]{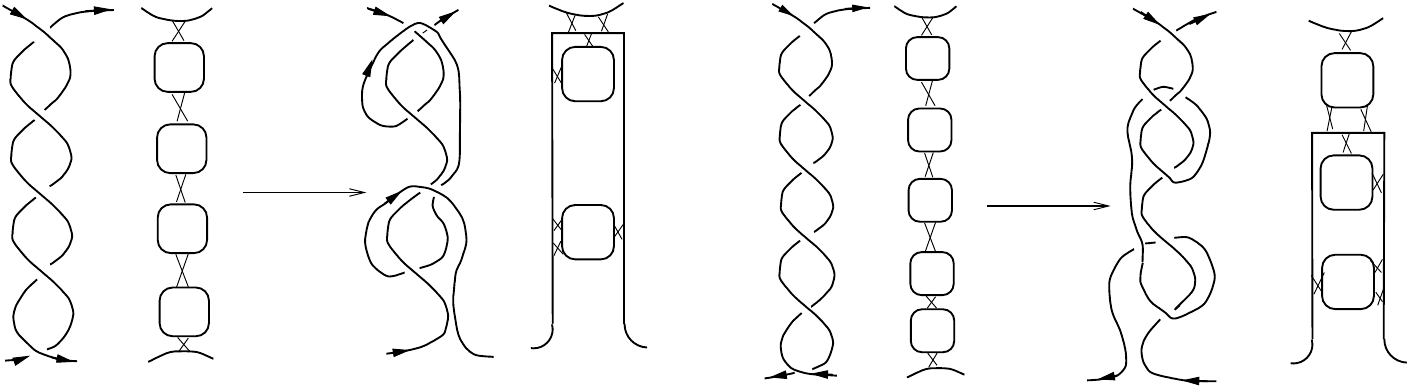}
\caption{The multiple MP moves performed on a string of Seifert circles connected by lone crossings. Left: The number of Seifert circles is reduced by $\lfloor \frac{4}{2}\rfloor=2$; Right: The number of Seifert circles is reduced by $\lfloor \frac{5}{2}\rfloor=2$.}\label{MP_move}
\end{figure}

\medskip
\begin{remark} \label{toruslink_formula} 
{\em
Let $T_o(2k, 2)$ and $T_o(-2k,2)$ be the torus links whose components have opposite orientations and whose crossings are positive and negative respectively, then   
\begin{eqnarray}
H(T_o(2k,2))&=&\left\{
\begin{array}{l}
z(a^{-1} +\ldots +a^{-2k+3})+(z+z^{-1})a^{-2k+1} -z^{-1}a^{-2k-1}, \ k>1,\\
(z+z^{-1})a^{-1} -z^{-1}a^{-3},\ k=1.
\end{array}
\right.\label{HTo(2k)}\\
H(T_o(-2k,2))&=&\left\{
\begin{array}{l}
-z(a +\ldots +a^{2k-3})-(z+z^{-1})a^{2k-1} +z^{-1}a^{2k+1}, \ k>1,\\
-(z+z^{-1})a +z^{-1}a^{3}, \ k=1.
\end{array}
\right.\label{HTo(-2k)}
\end{eqnarray}
On the other hand, if $T_p(n, 2)$ and $T_p(-n,2)$ ($n\ge 2$) are the torus knots/links whose components (when $n$ is even) have parallel orientations and whose crossings are positive and negative respectively, then   
\begin{eqnarray}
H(T_p(n,2))&=&
f_{n+2} a^{1-n} - f_{n} a^{-1-n},
\label{HTp(n)}\\
H(T_p(-n,2))&=&
(-1)^{n}( f_{n} a^{n+1}-f_{n+2} a^{n-1}),
\label{HTp(-n)}
\end{eqnarray}
where $\{f_k\}$ is the Fibonacci like sequence defined by $f_{n+2}=zf_{n+1}+f_n$, $f_2=z^{-1}$ and $f_3=1$. We have $\deg(f_n)=n-3$. 
It follows that if $D$ is the connected sum of $\rho^+_0$ positive torus links $T_p(\mu_j,2)$ and $\rho^-_0$ negative torus links $T_p(-\nu_i,2)$, then 
\begin{eqnarray}
E(D)&=&s(D)-w(D)-1,\quad p_0^h(D)=(-1)^{c^-(D)}z^{c(D)-\rho^+_0-3\rho^-_0},\label{TorusConnect1}\\
e(D)&=&-s(D)-w(D)+1,\quad  p_0^\ell(D)=(-1)^{\rho^+_0+\rho^-_0+c^-(D)}z^{c(D)-3\rho^+_0-\rho^-_0}.
\label{TorusConnect2}
\end{eqnarray}
If in addition, the connected sum components of $D$ also include $\kappa^+$ positive torus links $T_o(2\alpha_j,2)$ and $\kappa^-$ negative torus links $T_o(-2\beta_i,2)$, then 
\begin{eqnarray}
E(D)&=&s(D)-w(D)-1-2r^-(D),\quad p_0^h(D)\in (-1)^{c^-(D)}F,\label{TorusConnect3}\\
e(D)&=&-s(D)-w(D)+1+2r^+(D),\quad  p_0^\ell(D)\in (-1)^{\rho^+-\rho^-+\kappa^++\kappa^-+c^-(D)}F,
\label{TorusConnect4}
\end{eqnarray}
where $r^+(D)=-\kappa^++\sum \alpha_j$, $r^-(D)=-\kappa^-+\sum \beta_i$ and $F$ ($-F$) is the
set of all nonzero Laurent polynomials of $z$ whose coefficients are
non-negative (non-positive). 
}
\end{remark}

\medskip
\section{The cases when $\rho^++\rho^-\ge 2$, $\kappa^-=\kappa^+=0$}\label{Type3basic_case} 

\medskip
Before we tackle the general case, let us first understand the Type 3 pretzel links with $\kappa^+=\kappa^-=0$. Let us call such a Type 3 pretzel link a {\em basic} Type 3 pretzel link.  The following proposition has been established in the proof of \cite[Proposition 4.3]{Diao2021}.

\medskip
\begin{proposition}\cite{Diao2021}\label{P3.1}
If $D$ is a standard diagram of $\L\in P_3(\mu_1,\ldots,\mu_{\rho^+};0\vert 0;0)$ (so $\L$ is alternating), then 
\begin{eqnarray*}
E(\L)=2n-w(D)-1, &&p_0^h(\L)=z^{1+c(D)-2n}\in F,
\end{eqnarray*}  
and $e(\L)=-2n-w(D)+1+2\min\{n-1,\delta^+\}$ with
\begin{eqnarray*}
p_0^\ell(\L)&=&
\left\{
\begin{array}{ll}
(-1)^{1+\delta^+}z^{1+2\delta^++c(D)-6n} \in (-1)^{1+\delta^+}F &{\rm if}\ \delta^+<n-1;\\
-z^{c(D)-2n-1}\in -F &{\rm if}\ \delta^+\ge n-1.
\end{array}
\right.
\end{eqnarray*}
\end{proposition}

\medskip
\begin{proposition}\label{P3.2}
If $D$ is a standard diagram of $\L\in P_3(\mu_1,\ldots,\mu_{\rho^+};-\nu_1,\ldots,-\nu_{\rho^-}\vert 0;0)$ such that $\rho^-+\rho^+=2n\ge 4$ and $\delta^+=0$, then
\begin{eqnarray}
E(\L)&=&2n-w(D)-1,\ p_0^h(\L)\in(-1)^{c^-(D)}F\\
e(\L)&=&-2n-w(D)+1,\ 
p^\ell_0(\L)\in(-1)^{1+c^-(D)}F.
\end{eqnarray}
\end{proposition}

\begin{proof}
We can assume that $\rho^+>0$ and $\rho^->0$ since otherwise the result follows from Proposition \ref{P3.1}. Use induction on $\mu_{\rho^+}$, start with $\mu_{\rho^+}=2$ and apply VP$^{+}$ on a crossing in the $\mu_{\rho^+}$ strip. $D_-$ is the connected sum of $T_p(\mu_j,2)$, $1\le j\le \rho^+-1$ and $T_p(-\nu_i,2)$ ($1\le i\le \rho^-$) and $D_0$ has a positive lone crossing, hence by Remark \ref{toruslink_formula}  we have
\begin{eqnarray*}
-2+e(D_-)&=&-2-2n-(w(D)-2)+1=-2n-w(D)+1,\\ 
p_0^\ell(D_-)&\in &(-1)^{1+\sum \nu_i}F=(-1)^{1+c^-(D)}F.
\end{eqnarray*}
On the other hand, by \cite[Lemma 3.1]{Diao2021} (notice that the proof there applies to non-alternating link diagrams as well), $-1+e(D_0)\ge -1-s(D_0)-w(D_0)+3=-2n-w(D)+3>-2n-w(D)+1$ since $D_0$ has a positive lone crossing. It follows that $e(\L)=-2n-w(D)+1$ and
$
p_0^\ell(\L)=p_0^\ell(D_-)\in(-1)^{1+c^-(D)}F.
$
Similarly, if $\mu_{\rho^+}=3$, then $D_-$ has a positive lone crossing hence $-2+e(D_-)\ge  -2n-w(D)+3$, and we have $e(\L)=-1+e(D_0)=-2n-w(D)+1$ with $p_0^\ell(\L)=zp_0^\ell(D_0)\in(-1)^{1+c^-(D)}F.$ RLR for $\mu_{\rho^+}\ge 3$ in general. Here RLR stands for ``the rest is left to the reader".

\medskip
The proof for $E(\L)$ and $p^h_0(\L)$ is similar to the above by using induction on $\nu_1$ instead. RLR. 
\end{proof}

\medskip
The case $\rho^+=\rho^-=1$ is a special one and needs to be addressed separately.

\begin{proposition}\label{P3.3}
If $\rho^+=\rho^-=1$, then 
\begin{eqnarray}
&&E(\L)=e(\L)=0,\quad
p_0^h(\L)=p_0^\ell(\L)=1,\  {\rm if}\ |\mu_1-\nu_1|=1,\label{P3.3e1}\\
&&\left\{
\begin{array}{ll}
E(\L)&=2n-w(D)-1,\\
p^h_0(\L)&=
\left\{
\begin{array}{ll}
z^{\mu_1-\nu_1-1},& \ {\rm if}\ \mu_1-\nu_1>1\ {\rm or}\  \mu_1-\nu_1=0;\\
(-1)^{\mu_1+\nu_1}z^{-\mu_1+\nu_1-3},& \ {\rm if}\ \mu_1-\nu_1<-1,
\end{array}
\right.
\end{array}
\right.
\label{P3.3e2}\\
&&\left\{
\begin{array}{ll}
e(\L)&=-2n-w(D)+1,\\
p^\ell_0(\L)&=
\left\{
\begin{array}{ll}
-z^{\mu_1-\nu_1-3},& \ {\rm if}\ \mu_1-\nu_1>1;\\
(-1)^{1+\mu_1+\nu_1}z^{-\mu_1+\nu_1-1},& \ {\rm if}\ \mu_1-\nu_1<-1,\ {\rm or}\  \mu_1-\nu_1=0.
\end{array}
\right.
\end{array}
\right.\label{P3.3e3}
\end{eqnarray}
\end{proposition}

\begin{proof}
When $|\mu_1-\nu_1|=1$, $D$ is the trivial knot. When $|\mu_1-\nu_1|=0$, $D$ is the trivial link with $2$ components. In all other cases, $D$ reduces to the torus link $T_p(\mu_1-\nu_1,2)$, and the formulas follow from (\ref{HTp(n)}) and (\ref{HTp(-n)}).
\end{proof}

\medskip
We will now handle the other case not covered by Propositions \ref{P3.2} and \ref{P3.3}, namely the case where $\delta^+>0$, $\rho^->0$ and $2n=\rho^++\rho^-\ge 4$. We divide this case into two smaller cases: (i) $\delta^+\le \rho^-$ and (ii) $\delta^+>\rho^-$. We shall denote by $\tilde{D}$ the diagram obtained from $D$ by performing the moves (called the {\em N-moves}) as shown in Figure \ref{Type3fig}. 
\begin{figure}[htb!]
\includegraphics[scale=.9]{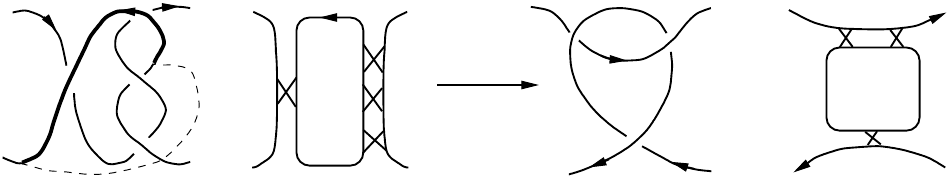}
\caption{ The depiction of an N-move: the strand to be re-routed is highlighted. An N-move changes the sum of two tangles $(+1)+(-\nu_j,0)$ to a rational tangle $(\nu_j-1,1)$.
}
\label{Type3fig} 
\end{figure}

\medskip
In case (i), we have $s(\tilde{D})=s(D)-\delta^+$, $w(\tilde{D})=w(D)-\delta^+$ and we claim the following.

\medskip
\begin{proposition}\label{P3.4}
Let $D$ be  a standard diagram of $\L\in P_3(\mu_1,\ldots,\mu_{\rho^+};-\nu_1,\ldots,-\nu_{\rho^-}\vert 0;0)$ such that $0<\delta^+\le \rho^-$. Let $\rho^+_0=\rho^+-\delta^+\ge 0$, $\rho^-_0=\rho^--\delta^+\ge 0$, $\rho^++\rho^-=2n$ (so $\rho^+_0+\rho^-_0=2n-2\delta^+$), then we have
\begin{eqnarray}
&E(\L)=e(\L)=0, \quad p^h_0(\L)=p^\ell_0(\L)=1,&{\rm if}\ \rho^+=\rho^-=\delta^+=1,\ \nu_1=2,\label{P3.4e3}\\
&\left\{\begin{array}{ll}
E(\L)=\nu_1, &p^h_0(\L)=(-1)^{\nu_1-1}z^{\nu_1-4}\\
e(\L)=\nu_1-2, &p^\ell_0(\L)=(-1)^{\nu_1}z^{\nu_1-2}
\end{array}
\right.
&{\rm if}\ \rho^+=\rho^-=\delta^+=1,\ \nu_1>2,\label{P3.4e4}\\
&
\left\{\begin{array}{ll}
E(\L)=2n-w(D)-1, &p^h_0(\L)\in(-1)^{c^-(D)}F,\\
e(\L)=-w(D)-1, &p^\ell_0(\L)\in(-1)^{1+\delta^++c^-(D)}F,
\end{array}
\right.
&{\rm if}\ \rho^+=\rho^-=\delta^+>1.\label{P3.4e5}
\end{eqnarray}
For other cases, we have
\begin{eqnarray}
&&
\left\{\begin{array}{ll}
E(\L)=2n-w(D)-1,&
p_0^h(\L)\in(-1)^{c^-(D)}F,\\
e(\L)=-2n-w(D)+1+2\delta^+, & p^\ell_0(\L)\in (-1)^{1+\delta^++c^-(D)}F.
\end{array}
\right.
\label{P3.4e1}
\end{eqnarray}
\end{proposition}

\begin{proof} If $\rho^+=\rho^-=\delta^+=1$, then $D$ is the unknot if $\nu_1=2$, otherwise $D$ is the torus link $T_p(-(\nu_1-1),2)$. Thus (\ref{P3.4e3}) and (\ref{P3.4e4}) hold. If $\rho^+=\rho^-=\delta^+\ge 2$, then 
$D$ is reduced to an alternating link diagram $\tilde{D}$. $\tilde{D}$ is a Type M2 link diagram as defined in \cite{Diao2021} and it has a negative lone crossing (there are multiple such crossings but only one can be used). We have $s(\tilde{D})=2+\delta^+$, $c(\tilde{D})=c(D)-\delta^+$ and $w(\tilde{D})=w(D)-\delta^+$. By the proof of \cite[Theorem 4.7]{Diao2021}, we have
\begin{eqnarray*}
E(\L)&=&E(\tilde{D})=s(\tilde{D})-w(\tilde{D})-3=2n-w(D)-1,\ p_0^h(\L)\in
(-1)^{c^-(D)}F,\\
e(\L)&=&e(\tilde{D})=-s(\tilde{D})-w(\tilde{D})+1=-w(D)-1,\ p_0^\ell(\L)\in
(-1)^{1+\delta^++c^-(D)}F.
\end{eqnarray*}
This proves (\ref{P3.4e5}).

\medskip
For the other cases,
we use double induction on $\delta^+\ge 1$ and $\nu_{\rho^-}\ge 2$, starting at $\delta^+= 1$ and $\nu_{\rho^-}= 2$. 
Apply VP$^-$ to a negative crossing in the $\nu_{\rho^-}$ strip.  $D_+$ is the connected sum of $T_p(\mu_j)$ ($1\le j\le \rho^+_0$) and $T_p(-\nu_i,2)$ ($1\le i\le \rho^--1$). Notice that $c(D_+)=c(D)-2-\delta^+$, $w(D_+)=w(D)+2-\delta^+$ and $s(D_+)=s(D)-\delta^+=2n-\delta^+$. Combining these with (\ref{TorusConnect1}) and (\ref{TorusConnect2}) (keep in mind that $\rho^-_0$ in (\ref{TorusConnect1}) is $\rho^--1$ here), 
we have  
\begin{eqnarray}
&&\left\{
\begin{array}{lll}
2+E(D_+)&=&2n-w(D)-1, \ p^h_0(D_+)\in (-1)^{c^-(D)}F,\\
2+e(D_+)&=&-2n-w(D)+1+2\delta^+, \  p^\ell_0(D_+) \in (-1)^{1+\delta^++c^-(D)}F.
\end{array}\right.\label{P3.4Torus}
\end{eqnarray}
Let us consider the case $\delta^+=1$ first.
The case $\rho^+_0=0$ and $\rho^-_0=0$ is the case $\rho^+=\rho^-=\delta^+=1$ already considered above. The remaining cases are (a) $\rho^+_0\ge 2$; (b) $\rho^-_0\ge 2$ and (c) $\rho^+_0=\rho^-_0=1$. 

\medskip
If $\rho^+_0\ge 2$ or $\rho^-_0\ge 2$, then either Proposition \ref{P3.1} applies to $D_0$ (in the case that $\rho^++\rho^-=4$) or Proposition \ref{P3.2} applies to $D_0$ (in the case that $\rho^++\rho^-\ge 6$). Actually, in the case that $\rho^+=1$ and $\rho^-=3$, we have to apply Proposition \ref{P3.1} to the mirror image of $D_0$. But in each case, we have
\begin{eqnarray*}
1+E(D_0)&=&1+(2n-2)-w(D_0)-1=2n-w(D)-3,\\
1+e(D_0)&=&1-(2n-2)-w(D_0)+1=-2n-w(D)+1+2\delta^+,\\
-zp^\ell_0(D_0)&\in&-z(-1)^{1+c^-(D_0)}F =(-1)^{c^-(D)}F=(-1)^{1+\delta^++c^-(D)}F.
\end{eqnarray*}
Direct comparison of the above with (\ref{P3.4Torus}) then leads to (\ref{P3.4e1}).
This proves the case of $\nu_{\rho^-}=2$. If $\nu_{\rho^-}=3$, $D_+$ is the same as $D_0$ in the case 
of $\nu_{\rho^-}=2$ hence $2+E(D_+)=2+(2n-2)-(w(D)+2)-1=2n-w(D)-3$ so $E(\L)=1+E(D_0)$ with 
$
p^h_0(\L)=-zp^h_0(D_0)
$
as given in (\ref{P3.4e1}).
 On the other hand, we again have $2+e(D_+)=1+e(D_0)=-2n-w(D)+3$, and comparison of $p^\ell_0(D_+)$ with $-zp^\ell_0(D_0)$ leads to $p^\ell_0(\L)=-zp^\ell_0(D_0)$ as given in (\ref{P3.4e1}).
Assuming now that the statement holds for $\nu_{\rho^-}\ge q\ge 3$, then for $\nu_{\rho^-}=q+1$, it is straight forward to see that $p^h_0(\L)$ and $p^\ell_0(\L)$ are both contributed by $D_0$ with the desired formulas for $p^h_0(\L)$ and $p^\ell_0(\L)$. 

\medskip
Now consider the remaining case $\rho^+_0=\rho^-_0=1$ (so $\rho^+=\rho^-=2$). Start with $\nu_{2}=2$ and apply VP$^-$ to a negative crossing in the $\nu_2$ strip. Proposition \ref{P3.3} applies to $D_0$, and  $D_+$ simplifies to $T_p(\mu_1,2)\#T_p(-\nu_1,2)$. Keep in mind that in this case $2n=4$, $w(D)=\mu_1-\nu_1-1$, $c(D)=\mu_1+\nu_1+3$. Notice also that in this case $c(D)-1-2n-2\rho^+_0=c(D)-2n-3=c(D)-7$.  We have
\begin{eqnarray*}
&2+E(D_+)=4-\mu_1+\nu_1=2n-w(D)-1, &p_0^h(D_+)=(-1)^{\nu_1}z^{\mu_1+\nu_1-4} \in (-1)^{c^-(D)}F,\\
&2+e(D_+)=-\mu_1+\nu_1=-2n-w(D)+3,&p_0^\ell(D_+)=(-1)^{\nu_1}z^{\mu_1+\nu_1-4}\in (-1)^{c^-(D)}F,
\end{eqnarray*}
and the following results concerning $D_0$:

\medskip\noindent
(i) $\mu_1-\nu_1=1$. In this case $D_0$ is the unknot and $w(D)=0$, hence $2+e(D_+)=-2n-w(D)+3=-1<1+e(D_0)=1+E(D_0)=1<3=2n-w(D)-1=2+E(D_+)$, so in this case (\ref{P3.4e1}) holds.\\
\noindent
(ii) $\mu_1-\nu_1=-1$. In this case $D_0$ is also the unknot and $w(D)=-2$. We have
\begin{eqnarray*}
1+E(D_0)&=&1<5=2+E(D_+)=2n-w(D)-1,\\
1+e(D_0)&=&1=2+e(D_+)=-2n-w(D)+3.
\end{eqnarray*}
On the other hand, $p^\ell_0(D_+)=(-1)^{\nu_1}z^{\mu_1+\nu_1-4}\in(-1)^{c^-(D)}F$ 
and $-zp^\ell(D_0)=-z$, so if $\mu_1\ge 3$, then $e(\L)=-2n-w(D)+3$ with $p^\ell_0(D_+)=p^\ell_0(\L)\in(-1)^{c^-(D)}F$.
If $\mu_1=2$, we have $p^\ell(D_+)=-z$ hence
$e(\L)=-2n-w(D)+3$ with 
$p^\ell_0(\L)=-2z\in (-1)^{1+\delta^++c^-(D)}F.$
So (\ref{P3.4e1}) holds.\\
\noindent
(iii) $\mu_1-\nu_1=0$. In this case $D_0$ is the trivial link with two components and $w(D)=-1$, hence $1+E(D_0)=2<4=2n-w(D)-1$. On the other hand, $1+e(D_0)=0=-2n-w(D)+3$ and $-zp^\ell_0(D_0)=1$. Thus if $c(D)>7$ then $e(\L)=-2n-w(D)+3$ with $p^\ell_0(\L)=p^\ell_0(D_+)\in(-1)^{c^-(D)}F$.
If  $c(D)=7$ then $\mu_1=\nu_1=2$ and $p^\ell_0(D_+)\in(-1)^{c^-(D)}F$ as well.
Thus we have $e(\L)=-2n-w(D)+3$ with 
$ p^\ell_0(\L)=2\in (-1)^{1+\delta^++c^-(D)}F$. 
Thus (\ref{P3.4e1}) holds.\\
\noindent
(iv) $\mu_1-\nu_1\ge 2$. In this case $D_0$ simplifies to $T_p(\mu_1-\nu_1,2)$ and $w(D)=\mu_1-\nu_1-1$, hence 
\begin{eqnarray*}
1+E(D_0)&=&1+2-(\mu_1-\nu_1)-1=2-\mu_1+\nu_1<4-\mu_1+\nu_1=2n-w(D)-1=2+E(D_+),\\
1+e(D_0)&=&1-2-(\mu_1-\nu_1)+1=-\mu_1+\nu_1=-2n-w(D)+3=2+e(D_+).
\end{eqnarray*}
Furthermore, we have
$
-zp^\ell_0(D_0)=z^{\mu_1-\nu_1-2}=z^{c(D)-2\nu_1-5}$. Since $\nu_1\ge 2$, $c(D)-2\nu_1-5\le c(D)-9<c(D)-7$ hence (\ref{P3.4e1}) holds with $E(\L)=2+E(D_+)=2n-w(D)-1$, $p^h(\L)=p^h(D^+)\in (-1)^{c^-(D)}F$, $e(\L)=-2n-w(D)+3$ and $p^\ell_0(\L)=p^\ell_0(D_+)\in (-1)^{c^-(D)}F=(-1)^{1+\delta^++c^-(D)}F$. \\
\noindent
(v) $\mu_1-\nu_1\le -2$. In this case $\nu_1-\mu_1\ge 2$ so $D_0$ simplifies to $T_p(-(\nu_1-\mu_1),2)$ and we still have 
\begin{eqnarray*}
1+E(D_0)&=&1+2-(\mu_1-\nu_1)-1=2-\mu_1+\nu_1<4-\mu_1+\nu_1=2n-w(D)-1=2+E(D_+),\\
1+e(D_0)&=&1-2-(\mu_1-\nu_1)+1=-\mu_1+\nu_1=-2n-w(D)+3=2+e(D_+).
\end{eqnarray*}
Thus the portion of (\ref{P3.4e1}) concerning $E(\L)$ holds. Furthermore, this time we have
$$
-zp^\ell_0(D_0)=(-1)^{\nu_1-\mu_1}z^{\nu_1-\mu_1}=(-1)^{\mu_1+c^-(D)}z^{c(D)-2\mu_1-3}.$$ 
If $\mu_1>2$, then $c(D)-2\mu_1-3<c(D)-7$ hence (\ref{P3.4e1}) holds. If $\mu_1=2$, then $c(D)-2\mu_1-3=c(D)-7$ hence $-zp^\ell_0(D_0)=(-1)^{1+\delta^++c^-(D)}z^{c(D)-7}=p^\ell_0(D_+)$. Thus $e(\L)=-2n-w(D)+3$ with $p^\ell_0(\L)=(-1)^{1+\delta^++c^-(D)}2z^{c(D)-7}=(-1)^{1+\delta^++c^-(D)}(1+\delta^+)z^{c(D)-7}$ so (\ref{P3.4e1}) holds.

\medskip
The above proves the case $\nu_2=2$. For $\nu_2= 3$, the above discussion applies to $D_0$ hence we have
\begin{eqnarray*}
1+E(D_0)&=&2n-w(D)-1,\ {\rm and} \ -zp^h_0(D_0)\in (-1)^{1+c^-(D_0)}F \\
1+e(D_0)&=&-2n-w(D)+3, \ {\rm and} \ -zp^\ell_0(D_0)\in (-1)^{c^-(D)}F.
\end{eqnarray*}
One needs to verify that $D_+=P_3(\mu_1,1;-\nu_1,-1\vert 0;0)=T_p(\mu_1-\nu_1,2)$ will not make contributions to the above. 
After this, we can use induction to prove the general case $\nu_2\ge 4$ and there it is fairly straight forward to see that $D_0$ always makes the sole contributions as we have demonstrated above. RLR.

\medskip
Finally, we use induction on $\delta^+$ to extend the result to any $\delta^+\ge 1$. 
Assuming that the statement of the proposition holds for some $\delta^+-1\ge 1$ and consider the case of $\delta^+$.
Apply VP$^-$ to a negative crossing in the $\nu_{\rho^-}$ strip. If $\nu_{\rho^-}=2$ then (\ref{P3.4Torus}) applies to $D_+$ and the induction hypothesis applies to $D_0$ which has $2n-2$ Seifert circles with $\delta^+-1$ lone crossings and $c^-(D_0)=c^-(D)-2$. We have
\begin{eqnarray*}
1+E(D_0)&=&1+(2n-2)-w(D_0)-1=2n-w(D)-3,\\
1+e(D_0)&=&1-(2n-2)-w(D_0)+1+2(\delta^+-1)=-2n-w(D)+1+2\delta^+,\\
-zp^\ell_0(D_0)&\in& (-1)^{1+\delta^++c^-(D)}F.
\end{eqnarray*}
We see that $E(\L)=2+E(D_+)=2n-w(D)-1$ with $p^h_0(\L)= p^h_0(D_+)\in(-1)^{c^-(D)}F$
by (\ref{TorusConnect1}). On the other hand, $2+e(D_+)=-2n-w(D)+1+2\delta^+$ with $p^\ell_0(D_+)\in (-1)^{1+\delta^+ + c^-(D)}$
by (\ref{P3.4Torus}). Thus,  $p^\ell_0(D_+)$ and $-zp^\ell_0(D_0)$ will both make a contribution to $p^\ell(\L)$ only when $\rho^+_0=\rho^-_0=1$. Otherwise only $-zp^\ell_0(D_0)$ makes a contribution to $p^\ell(\L)$ . This results in the second portion of (\ref{P3.4e1}). After this, we can use induction to prove the general case $\nu_{\rho^-}\ge 3$ and there it is fairly straight forward to see that $D_0$ always makes the sole contributions to the $E$ and $e$ powers. RLR.
\end{proof}

\medskip
We now consider the case $\delta^+>\rho^-$. 

\medskip
\begin{proposition}\label{P3.5}
Let $D$ be  a standard diagram of $\L\in P_3(\mu_1,\ldots,\mu_{\rho^+};-\nu_1,\ldots,-\nu_{\rho^-}\vert 0;0)$ such that $\delta^+>\rho^-$,   then we have
\begin{eqnarray}
E(\L)&=&2n-w(D)-1,\ 
p_0^h(\L)=
(-1)^{c^-(D)}F,\label{P3.5e1}
\end{eqnarray}

\begin{eqnarray}
\left\{
\begin{array}{lll}
e(\L)&=&-2n-w(D)+1+2\min\{\delta^+, n-1\},\\
p^\ell_0(\L)&\in&\left\{
\begin{array}{ll}
(-1)^{\rho^-+c^-(D)}F,&{\rm if}\ \delta^+<n-1,\\
(-1)^{1+\rho^-+c^-(D)}F, &{\rm if} \ \delta^+\ge n-1.
\end{array}
\right.
\end{array}
\right.\label{P3.5e2}
\end{eqnarray}
\end{proposition}

\begin{proof} We first use the $N$-moves to change $D$ to a Type B alternating link diagram $\tilde{D}$ as defined in \cite[Theorem 4.1]{Diao2021}. We have $s(\tilde{D})=s(D)-\rho^-$, $w(\tilde{D})=w(D)-\rho^-$. The Seifert circle decomposition of $\tilde{D}$ contains a cycle of $2n-2\rho^-=\rho^+-\rho^-$ Seifert circles, connected by $\mu_1$, ..., $\mu_{\rho^+_0}$ positive crossings and $\delta^+-\rho^-$ lone crossings (where $\rho_0^+=\rho^+-\delta^+$), one of these Seifert circles in the cycle also has $\rho^-$ Seifert circles attached to it by $\nu_1$, ..., $\nu_{\rho^-}$ negative crossings. By \cite[Theorem 4.1]{Diao2021} and its proof, we have
$$
E(\L)=E(\tilde{D})=s(\tilde{D})-w(\tilde{D})-1=s(D)-w(D)-1,
$$
and $p_0^h(\L)=p_0^h(\tilde{D})=z^{\delta^+-\rho^-}p_0^h(\hat{D})$, where $\hat{D}$ is the diagram obtained from $\tilde{D}$ by smoothing all the lone crossings in it. $\hat{D}$ is the connected sum of $T_p(\mu_j,2)$, $1\le j\le \rho^+-\delta^+$, $D_1$, plus a disjoint union of $\delta^+-\rho^--1$ trivial knots, where $D_1$ is the alternating link whose Seifert circle decomposition has the structure of one Seifert circle $C$ with $\rho^-$ Seifert circles attached to it (either from inside or outside of $C$) by $\nu_1$, ..., $\nu_{\rho^-}$ negative crossings. Notice that Remark \ref{elementary_remark} applies to $D_1$. Direct calculation then leads to 
$$
p^h_0(\hat{D})=\big(z^{-\rho^++\sum_{i=1}^{\rho^+-\delta^+}\mu_i}\big)\big((-1)^{c^-(D_1)}z^{c^-(D_1)-3\rho^-}\big)\big(z^{-\delta^++\rho^-+1}\big)
$$
and 
$$
p_0^h(\tilde{D})=z^{\delta^+-\rho^-}p_0^h(\hat{D})=(-1)^{c^-(D)}z^{1+c(D)-2n-2\rho^-}\in (-1)^{c^-(D)}F,
$$
which is (\ref{P3.5e1}). On the other hand, \cite[Theorem 4.1]{Diao2021} asserts that 
 \begin{eqnarray*}
e(\L)=e(\tilde{D})&=&-s(\tilde{D})-w(\tilde{D})+1+2\min\{\delta^+-\rho^-,n-1-\rho^-\}\\
&=&-s(D)-w(D)+1+2\rho^-+2\min\{\delta^+-\rho^-,n-1-\rho^-\}\\
&=&-s(D)-w(D)+1+2\min\{\delta^+,n-1\}
\end{eqnarray*}
with
\begin{eqnarray*}
p^\ell_0(\L)=p^\ell_0(\tilde{D})&=&\left\{
\begin{array}{ll}
-zp^\ell_0(\hat{D}),&{\rm if}\ \delta^+<n-1,\\
z^{-1+3(n-\rho^-)}p^\ell_0(\hat{D}), &{\rm if} \ \delta^+= n-1,\\
z^{-2-(\delta^+-\rho^-)+4(n-\rho^-)}p^\ell_0(\hat{D}), &{\rm if} \ \delta^+> n-1.\\
\end{array}
\right.
\end{eqnarray*}
Using Remarks \ref{elementary_remark} and \ref{toruslink_formula} we can show by direct calculation that 
\begin{eqnarray*}
p^\ell_0(\hat{D})&=&\big((-1)^{\rho^+-\delta^+}z^{\sum_{i=1}^{\rho^+-\delta^+}\mu_i-3(\rho^+-\delta^+)}\big)\big((-1)^{c^-(D_1)+\rho^-}z^{c^-(D_1)-\rho^-}\big)\big((-z)^{-\delta^++\rho^-+1}\big)\\
&=&(-1)^{1+\rho^-+c^-(D)}z^{1+\delta^++c(D)-3\rho^+} \in (-1)^{1+\rho^-+c^-(D)}F.
\end{eqnarray*}
Substituting this into the above formula for $p^\ell_0(\tilde{D})$ and (\ref{P3.5e2}) follows. We leave the details of the calculations to the reader. 
\end{proof}

Combining Propositions \ref{P3.1} to \ref{P3.5}, we have (for any $\L\in P_3(\mu_1,\ldots,\mu_{\rho^+};-\nu_1,\ldots,-\nu_{\rho^-}\vert 0;0)$):
\begin{eqnarray}
E(\L)&=&
\left\{
\begin{array}{lll}
0,& {\rm if}\ \rho^+=\rho^-=1, |\mu_1-\nu_1|=1,\\
2n-w(D)-1,& {\rm otherwise},\\
\end{array}
\right.\label{Basic_e1}\\
e(\L)&=&
\left\{
\begin{array}{lll}
0,& {\rm if}\ \rho^+=\rho^-=1, |\mu_1-\nu_1|=1,\\
-w(D)-1,& {\rm if}\ \rho^+=\rho^-=\delta^+>1,\\
-2n-w(D)+1+2\min\{\delta^+,n-1\},&  {\rm otherwise},
\end{array}
\right.\label{Basic_e2}
\end{eqnarray}
and
\begin{eqnarray}
\b_0(\L)&=&\frac{E(\L)-e(\L)}{2}+1=
\left\{
\begin{array}{ll}
0,& {\rm if}\ \rho^+=\rho^-=1, |\mu_1-\nu_1|=1,\\
n+1,& {\rm if}\ \rho^+=\rho^-=\delta^+>1,\\
2n-\min\{\delta^+,n-1\},&  {\rm otherwise}.
\end{array}
\right.\label{Basic_e3}
\end{eqnarray}

\section{The cases when $\rho^+=\rho^-=1$, $\kappa^++\kappa^->0$}\label{Type3S2}

\medskip
\begin{proposition}\label{P4.1}
Let $D$ be  a standard diagram of $\L\in P_3(\mu_1;-\nu_1\vert 2\alpha_1,\ldots, 2\alpha_{\kappa^+};0)$ such that $\mu_1>1$ and $\nu_1=\mu_1+1$, then we have
\begin{eqnarray}
&e(\L)=-\kappa^+-2\sum \alpha_j, &E(\L)=
\left\{
\begin{array}{ll}
2-\kappa^+, &{\rm if}\ \alpha_{\kappa^+}=1,\\
-\kappa^+, &{\rm if}\ \alpha_{\kappa^+}>1,
\end{array} 
\right.\label{exception1}\\
&p^\ell_0(\L)\in (-1)^{1+\kappa^++c^{-}(D)}F, &\left\{
\begin{array}{ll}p^h(\L)=z,  &{\rm if}\ \kappa^+=1, \mu_1=2 \ {\rm and}\ \alpha_1> 2,\\ 
p^h_0(\L) \in (-1)^{c^{-}(D)}F, &{\rm otherwise}.
\end{array}
\right.\label{exception2}
\end{eqnarray}
Similarly, if $D$ is a standard diagram of $\L\in P_3(\mu_1;-\nu_1\vert 0;-2\beta_1,\ldots, -2\beta_{\kappa^-})$ such that $\mu_1=\nu_1+1$, then we have
\begin{eqnarray}
&E(\L)=\kappa^-+2\sum \beta_i, &e(\L)=
\left\{
\begin{array}{ll}
\kappa^--2, &{\rm if}\ \beta_{\kappa^-}=1,\\
\kappa^-, &{\rm if}\ \beta_{\kappa^-}>1,
\end{array}
\right.\label{exception3}\\
&p^h_0(\L)\in (-1)^{c^{-}(D)}F, &\left\{
\begin{array}{ll}
p^\ell_0(\L)=-z, & {\rm if}\ \kappa^-=1, \nu_1=2 \ {\rm and}\ \beta_1> 2,\\ 
p^\ell_0(\L)\in (-1)^{1+\kappa^-+c^{-}(D)}F, &{\rm otherwise}.
\end{array}
\right.\label{exception4}
\end{eqnarray}
\end{proposition}

\begin{proof}
We will only prove (\ref{exception1}) and (\ref{exception2}) since  (\ref{exception3}) and (\ref{exception4}) can be obtained from (\ref{exception1}) and (\ref{exception2}) using the mirror image of $\L$. For $\alpha_1=1$, apply VP$^+$ to a crossing in the $\alpha_1$ strip. $D_0$ is the unknot (hence $-1+E(D_0)=-1+e(D_0)=-1$) and $D_-=T_p(\mu_1,2)\# T_p(-(\mu_1+1),2)$. It follows that 
\begin{eqnarray*}
E(\L)&=&-2+E(D_-)=1,\\
p^h(\L)&=&p^h(D_-)=(-1)^{\mu_1+1}f_{\mu_1+2}f_{\mu_1+1}=(-1)^{c^-(D)}f_{\mu_1+2}f_{\mu_1+1},\\
e(\L)&=&-2+e(D_-)=-3, \\
 p^\ell_0(\L)&=&p^\ell_0(D_-)=(-1)^{c^-(D)}z^{2\mu_1-3}.
\end{eqnarray*}
Notice that $p_0^h(\L)=(-1)^{c^-(D)}z^{2\mu_1-3}$.

\medskip
If $\alpha_1=2$, then $D_0$ is again the unknot and $D_-$ corresponds to the case $\alpha_1=1$. $-2+E(D_-)=-2+1=-1$ with $p^h_0(D_-)=(-1)^{c^-(D)}z^{c(D)-8}$. If $\mu_1> 2$, then $c(D)-8=2\mu_1-3>1$ hence $E(\L)=-1$ with $p^h_0(\L)=(-1)^{c^-(D)}z^{c(D)-8}$. If $\mu_1=2$, then $c(D)-8=1$ and $\nu_1=c^-(D)=3$ hence $p^h(D_-)=-f_{4}f_{3}=-(z+z^{-1})$. Thus $p^h(\L)=p^h(D_-)+zp^h(D_0)=-z^{-1}=(-1)^{c^-(D)}z^{c(D)-10}$. Either way, we have $E(\L)=-1$ with $p^h_0(\L)\in (-1)^{c^-(D)}F$. On the other hand, $-2+e(D_-)=-2-3=-5$ so $e(\L)=-5$ with $p^\ell_0(\L)=p^\ell_0(D_-)\in (-1)^{c^-(D)}F$. From here it is easy to see that in general, for $\alpha_1\ge 3$ we have 
\begin{eqnarray*}
E(\L)&=&-1+E(D_0)=-1=\kappa^+-2=-\kappa^+,\ p^h(\L)=zp^h(D_0)=z\in F,\\
e(\L)&=& -1-2\alpha_1=-\kappa^+-2\alpha_1, \ p^\ell_0(\L)\in (-1)^{c^-(D)}F=(-1)^{1+\kappa^++c^-(D)}F.
\end{eqnarray*}
Combining this with the cases of $\alpha_1=1$ and $\alpha_1=2$ proves the case $\kappa^+=1$. 

\medskip
Now assume that the statement holds for $\kappa^+=q-1\ge 1$ and consider the case $\kappa^+=q$. Start with $\alpha_{\kappa^+}=1$. $D_0$ corresponds to ${\kappa^+}-1$ where the induction hypothesis applies and $D_-$ simplifies to the connected sum of $T_o(2\alpha_j,2)$ ($1\le j\le \kappa^+-1$), $T_p(\mu_1,2)$ and $T_p(-(\mu_1+1),2)$. We have
\begin{eqnarray*}
-1+e(D_0)&=&-1-(\kappa^+-1)-2\sum_{j\le \kappa^+-1} \alpha_j=-\kappa^+-2\sum_{j\le \kappa^+-1} \alpha_j, \\
-2+e(D_-)&=&
-2-\kappa^+-2\sum_{j\le \kappa^+-1} \alpha_j=-\kappa^+-2\sum \alpha_j.
\end{eqnarray*}
It follows that $e(\L)=-2+e(D_-)=-\kappa^+-2\sum \alpha_j$ with $p^\ell(\L)=p^\ell(D_-)\in (-1)^{1+\kappa^++c^{-}(D)}F$. On the other hand,
\begin{eqnarray*}
-1+E(D_0)&=&
\left\{
\begin{array}{ll}
-1+2-(\kappa^+-1)=2-\kappa^+, &{\rm if}\ \alpha_{\kappa^+-1}=1,\\
-1-(\kappa^+-1)=-\kappa^+, &{\rm if}\ \alpha_{\kappa^+-1}>1,
\end{array}
\right.\\
-2+E(D_-)&=&2-\kappa^+.
\end{eqnarray*}
Thus $E(\L)=-2+E(D_-)=2-\kappa^+$ with $p^h(\L)=p^h(D_-)\in (-1)^{c^{-}(D)}F$ if $\alpha_{\kappa^+-1}>1$. If $\alpha_{\kappa^+-1}=1$, we have
$zp^h_0(D_0)\in (-1)^{c^{-}(D)}F$ by the induction hypothesis and 
$p^h_0(D-)\in (-1)^{c^{-}(D)}F$ by direction computation. Thus $E(\L)=2-\kappa^+$ with $p^h(\L)\in (-1)^{c^{-}(D)}F$.
So the result follows for $\alpha_{\kappa^+}=1$. If $\alpha_{\kappa^+}>1$, then $\alpha_{\kappa^+-1}>1$ and the induction follows trivially. 
\end{proof}

\medskip
\begin{proposition}\label{P4.2}
Let $D$ be  a standard diagram of $\L\in P_3(\mu_1;-\nu_1\vert 2\alpha_1,\ldots, 2\alpha_{\kappa^+};-2\beta_1,\ldots,-2\beta_{\kappa^-})$ such that one of the following conditions holds: {\em (a)} $\mu_1>1$, $\nu_1=\mu_1+1$ and $\kappa^->0$; {\em (b)} $\nu_1>1$, $\mu_1=\nu_1+1$ 
and $\kappa^+>0$; {\em (c)} $\mu_1>1$, $|\mu_1-\nu_1|\not=1$ (and $\nu_1>1$), then
\begin{eqnarray}
E(\L)&=&s(D)-w(D)-1-2r^-(D)=1+\nu_1-\mu_1+\kappa^--\kappa^++2\sum \beta_i,\label{P4.2e1}\\
e(\L)&=&-s(D)-w(D)+1+2r^+(D)=-1+\nu_1-\mu_1+\kappa^--\kappa^+-2\sum\alpha_j,\label{P4.2e2}
\end{eqnarray}
where $r^-(D)=-\kappa^-+\sum \beta_i$ and $r^+(D)=-\kappa^++\sum \alpha_j$. 
\end{proposition}
\begin{proof}
We shall only prove (\ref{P4.2e1}) and (\ref{P4.2e2}) under conditions (a) and (c), since (b) can be  obtained from (a) using the mirror image of $\L$. Let us consider (a) first.

\medskip
Case 1: $\kappa^+=0$. Here we shall prove (\ref{P4.2e1}) and (\ref{P4.2e2}) together with the claim that
$p^\ell_0(\L)\in (-1)^{1+\kappa^-+c^-(D)}F$ if $\beta_{\kappa^-}=1$ and $p^\ell(\L)= (-1)^{\kappa^-}z^{\kappa^-}$ if $\beta_{\kappa^-}>1$. Use induction on $\rho^-$, starting with $\rho^-=1$ and $\beta_1=1$. Apply VP$^-$ to a negative crossing in the $\beta_1$ strip. $D_+$ is the connected sum of $T_p(\mu_1,2)$ and $T_p(-(\mu_1+1),2)$ and $D_0$ is the unknot (hence $1+E(D_0)=1+e(D_0)=1$). By (\ref{TorusConnect1}) and (\ref{TorusConnect2}), we have $2+E(D_+)=5=2+\kappa^-+2\beta_1$, so (\ref{P4.2e1}) holds.
On the other hand, $2+e(D_+)=1=1+e(D_0)$ with $-zp^\ell(D_0)=-z$ and $p^\ell_0(D_+)=(-1)^{c^-(D)}z^{2\mu_1-3}$. Since $2\mu_1-3\ge 1$ with equality holding only when $\mu_1=2$, and $(-1)^{c^-(D)}=(-1)^5=-1$ when $\mu_1=2$, we see that $e(\L)=1=\kappa^-$ with $p^\ell_0(\L)\in (-1)^{c^-(D)}F=(-1)^{1+\kappa^-+c^-(D)}F$. This proves the case $\beta_1=1$. For $\beta_1=2$, again apply VP$^-$ to a negative crossing in the $\beta_1$ strip. It is easy to see that $E(\L)=2+E(D_+)=7=2+\kappa^-+2\beta_1$, and  $e(\L)=1+e(D_0)=1=\kappa^-<2+e(D_+)=3$, hence $p^\ell(\L)=-zp^\ell(D_0)=-z$. RLR for $\beta_1\ge 2$ in general. Assume now that the claim holds for $1\le \kappa^-\le q$ for some $q\ge 1$, consider the case $\kappa^-=q+1$, starting from $\beta_{q+1}=1$. The induction hypothesis applies to $D_0$ and $D_+$ is the connected sum of $T_p(\mu_1,2)$, $T_p(-(\mu_1+1),2)$, and $T_o(-2\beta_i,2)$, $1\le i\le q$. By the induction hypothesis, (\ref{TorusConnect3}) and (\ref{TorusConnect4}), we have 
\begin{eqnarray*}
2+E(D_+)&=&5+q+2\sum_{1\le i\le q}\beta_i=2+\kappa^-+2\sum_{1\le i\le \kappa^-} \beta_i\\
>1+E(D_0)&=&3+q+2\sum_{1\le i\le q}\beta_i.
\end{eqnarray*}
Thus (\ref{P4.2e1}) holds. On the other hand, we have
$2+e(D_+)=1+q=\kappa^-=1+e(D_0)$ and $p^\ell_0(D_+)=(-1)^{q+c^-(D)}z^{q+2\mu_1-3}$. If $\beta_q=1$, then $-zp_0^\ell(D_0)\in (-1)^{q+c^-(D)}F$ by the induction hypothesis, and we have $e(\L)=\kappa^-$ with $p_0^\ell(\L)\in (-1)^{q+c^-(D)}F=(-1)^{1+\kappa^-+c^-(D)}F$. If $\beta_q>1$, then $-zp^\ell(D_0)=(-1)^{q+1}z^{q+1}$ by the induction hypothesis. Since $q+2\mu_1-3\ge q+1$ with equality holding only when $\mu_1=2$, and $c^-(D)$ is odd when $\mu_1=2$, we see that $e(\L)=\kappa^-$ with $p_0^\ell(\L)\in (-1)^{1+\kappa^-+c^-(D)}F$. 
For $\beta_{q+1}=2$, again apply VP$^-$ to a negative crossing in the $\beta_{q+1}$ strip. Keep in mind that we have $\beta_q\ge \beta_{q+1}\ge 2$. It is easy to see that $2+E(D_+)=2+\kappa^-+2\sum_{1\le i\le \kappa^-} \beta_i>1+E(D_0)$ so (\ref{P4.2e1}) holds. But this time we have $1+e(D_0)=q+1<2+e(D_+)=q+3$, hence we have $e(\L)=1+e(D_0)=q+1=\kappa^-$ with $p^\ell(\L)=-zp^\ell(D_0)=(-1)^{\kappa^-}z^{\kappa^-}$ by the induction hypothesis. RLR for $\beta_{q+1}\ge 2$ in general. This completes our induction.

\medskip
Case 2: $\kappa^+>0$. Here we shall prove (\ref{P4.2e1}) and (\ref{P4.2e2}) together with the claim that
$p^\ell_0(\L)\in (-1)^{1+\kappa^-+\kappa^++c^-(D)}F$. Use induction on $\rho^-$, starting with $\rho^-=1$ and $\beta_1=1$. Apply VP$^-$ to a negative crossing in the $\beta_1$ strip. Proposition \ref{P4.1} applies to $D_0$ and $D_+$ is the connected sum of $T_p(\mu_1,2)$, $T_p(-(\mu_1+1),2)$, and $T_o(2\alpha_j,2)$ ($1\le j\le \kappa^+$). By (\ref{TorusConnect3}) and (\ref{TorusConnect4}), we have (using $r^+(D_+)=-\kappa^++\sum \alpha_j$)
\begin{eqnarray*}
2+E(D_+)&=&5-\kappa^+>3-\kappa^+\ge 1+E(D_0),\\
2+e(D_+)&=&1-\kappa^+-2\sum \alpha_j,\ p^\ell_0(D_+)\in (-1)^{\kappa^++c^-(D_+)}F=(-1)^{1+\kappa^-+\kappa^++c^-(D)}F,\\
1+e(D_0)&=&1-\kappa^+-2\sum \alpha_j,\ -zp^\ell_0(D_0)\in (-1)^{1+\kappa^-+\kappa^++c^-(D)}F.
\end{eqnarray*}
So our claim holds. From here it is straight forward to use induction to show that (\ref{P4.2e1}) and (\ref{P4.2e2}) hold for any $\beta_1\ge 1$ with the claim that $p_0^\ell(\L)\in (-1)^{1+\kappa^-+\kappa^++c^-(D)}F$. RLR for $\kappa^-\ge 1$ in general.
  
\medskip
We now prove (\ref{P4.2e1}) and (\ref{P4.2e2}) under condition (c).
Under this condition, a link in $P_3(\mu_1;-\nu_1\vert 0;0)$ is equivalent to $T_p(\mu_1-\nu_1,2)$ when $\mu_1\not=\nu_1$, or the trivial link with two components if $\mu_1=\nu_1$. The proof is done by dividing the case into several smaller cases. We shall leave most of the details in the proof to the reader as most of these are tedious calculations.

\medskip\noindent
Case 1. $\kappa^+\ge 1$, $\kappa^-=0$. Here we have $s(D)=2+2\sum \alpha_j-\kappa^+$, $w(D)=2\sum \alpha_j+\mu_1-\nu_1$ and $r^-(D)=0$. Here we prove (\ref{P4.2e1}) and (\ref{P4.2e2}) together with the claim that 
$p^h_0(\L)\in  (-1)^{c^{-}(D)}F$ and  $p^\ell_0(\L)\in (-1)^{1+\kappa^++c^{-}(D)}F$, by using double induction on $\kappa^+$ and $\alpha_{\kappa^+}$. Apply VP$^+$ to a crossing in the $\alpha_1$ strip. $D_-=T_p(\mu_1,2)\# T_p(-\nu_1,2)$, $D_0=T_p(\mu_1-\nu_1,2)$. It follows that 
\begin{eqnarray*}
-2+E(D_-)&=&-1+E(D_0)=\nu_1-\mu_1=s(D)-w(D)-1,\\
-2+e(D_-)&=&\nu_1-\mu_1-4=-s(D)-w(D)+1<-1+e(D_0)=\nu_1-\mu_1-2,\\
p^h_0(D_-)&=&p^\ell_0(D_-)=(-1)^{\nu_1}z^{\mu_1+\nu_1-4},\\
zp^h(D_0)&=& \left\{
\begin{array}{ll}
z^{\mu_1-\nu_1},& {\rm if}\ \mu_1\ge \nu_1,\\
(-1)^{\nu_1-\mu_1}z^{\nu_1-\mu_1-2},& {\rm if}\ \mu_1< \nu_1.
\end{array}
\right.
\end{eqnarray*}
Thus $e(\L)=-s(D)-w(D)+1$ with $p^\ell_0(\L)=p^\ell_0(D_-)=(-1)^{\nu_1}z^{\mu_1+\nu_1-4}\in (-1)^{1+\kappa^++c^-(D)}F$. On the other hand, $\mu_1+\nu_1-4\ge \mu_1-\nu_1$ with equal sign holding only when $\nu_1=2$ (keep in mind that we only consider the case $\nu_1>1$) and $\mu_1+\nu_1-4> \nu_1-\mu_1-2$, we have $E(\L)=s(D)-w(D)-1$ with $p^h_0(\L)=(-1)^{c^-(D)}2z^{\mu_1+\nu_1-4}$ for $2=\nu_1<\mu_1$ and $p^h_0(\L)=(-1)^{c^-(D)}z^{\mu_1+\nu_1-4}$ otherwise. This proves the case $\alpha_1=1$. RLR for $\alpha_1\ge 1$ by induction on $\alpha_1$ and this proves the case for $\kappa^+=1$. 

\medskip
Assuming that the statement holds for $1\le \kappa^+\le q$ for some $q\ge 1$ and consider $\kappa^+=q+1$, start with $\alpha_{\kappa^+}=1$. The induction hypothesis applies to $D_0$ (with $s(D_0)=s(D)-1$, $w(D_0)=w(D)-2$ and $r^+(D_0)=r^+(D)$) and $D_-$ is the connected sum of $T_o(2\alpha_j,2)$'s ($1\le j\le q$), $T_p(\mu_1,2)$ and $T_p(-\nu_1,2)$. It follows that 
\begin{eqnarray*}
-2+E(D_-)&=&-1+E(D_0)=s(D)-w(D)-1,\\
zp^h_0(D_0)&\in&(-1)^{c^{-}(D_0)}F=(-1)^{c^{-}(D)}F,\\
p^h_0(D_-)&\in&(-1)^{c^{-}(D_-)}F=(-1)^{c^{-}(D)}F,\\
-2+e(D_-)&=&-s(D)-w(D)+1+2r^+(D)<-1+e(D_0)=-s(D)-w(D)+3+2r^+(D),\\
p^\ell_0(D_-)&=&(-1)^{1-\kappa^++\nu_1}z^{\mu_1+\nu_1-\kappa^+-3}\in (-1)^{1+\kappa^++c^-(D)}F.
\end{eqnarray*}
The result follows. RLR  for the general case $\alpha_{\kappa^+}>1$ by induction on $\alpha_{\kappa^+}>1$. This completes the proof for Case 1.

\medskip\noindent
Case 2. $\kappa^+=0$, $\kappa^-\ge 1$. In this case $s(D)=2+2\sum \beta_i -\kappa^-$, $w(D)=\mu_1-\nu_1-2\sum \beta_i$, $r^+(D)=0$ and $r^-(D)=\sum \beta_i-\kappa^-$. Here we prove (\ref{P4.2e1}) and (\ref{P4.2e2}) together with the claim that 
$p^h_0(\L)\in  (-1)^{c^{-}(D)}F$ and 
$p^\ell_0(\L)\in (-1)^{1+\kappa^-+c^{-}(D)}F$
by double induction on $\kappa^-$ and $\beta_{\kappa^-}$. This case is similar to Case 1 and is left to the reader. 

\medskip\noindent
Case 3. $\kappa^+\ge 1$ and $\kappa^-\ge 1$. Here we prove (\ref{P4.2e1}) and (\ref{P4.2e2}) together with the claim that $p^h_0(\L)\in  (-1)^{c^{-}(D)}F$ and 
$p^\ell_0(\L)\in (-1)^{1+\kappa^-+\kappa^++c^{-}(D)}F$ using double induction on $\kappa^+$ and $\alpha_{\kappa^+}$. In this case the initial step $\kappa^+=0$ of the induction process has been established in Case 2. We leave the details to the reader. This completes the proof for Proposition \ref{P4.2}.
\end{proof}

 \medskip
 \begin{remark}\label{R4.3}{\em 
 We have now considered all possible cases under the condition $\mu_1>1$ in Propositions \ref{P4.1} and  \ref{P4.2}. The remaining cases are under the condition $\mu_1=1$. We note that if $\L\in P_3(1;-2\vert 2\alpha_1,\ldots, 2\alpha_{\kappa^+}; -2\beta_1,\ldots, -2\beta_{\kappa^-})$, then $\L\in P_2(2\alpha_1,\ldots, 2\alpha_{\kappa^+}; -2\beta_1,\ldots, -2\beta_{\kappa^-},-2)$. Thus $\L$ is the unknot if $\kappa^+=\kappa^-=0$, $\L=T_o(2(\alpha_1-1),2)$ if $\kappa^+=1$, $\kappa^-=0$ and $\L=T_o(2(\beta_1+1),2)$ if $\kappa^-=1$, $\kappa^+=0$, while the other cases of $\L$ can be obtained from \cite[Theorem 1.4]{Diao2024}. 
 We will now assume $\nu_1\ge 3$.
 }
 \end{remark}
 
\medskip
\begin{proposition}\label{P4.4} Let $\L\in P_3(1;-\nu_1\vert 2\alpha_1,\ldots,2\alpha_{\kappa^+};0)$, then
\begin{eqnarray}
E(\L)&=&
\left\{
\begin{array}{ll}
\nu_1-\kappa^+, &{\rm if}\ 
\left\{
\begin{array}{l}
\nu_1=3, \kappa^+\ge 2, \alpha_{\kappa^+}=1, \alpha_{\kappa^+-1}=1,\ {\rm or}\\
\nu_1=3, \kappa^+\ge 1, \alpha_{\kappa^+}>1, \ {\rm or}\\
\nu_1\ge 4,
\end{array}
\right.\\
\nu_1-\kappa^+-2, &{\rm if}\ 
\left\{
\begin{array}{l}
\nu_1=3, \kappa^+\ge 2, \alpha_{\kappa^+}=1, \alpha_{\kappa^+-1}>1,\ {\rm or}\\
\nu_1=3, \kappa^+=1, \alpha_{\kappa^+}=1,
\end{array}
\right.
\end{array}
\right.\label{P4.4e1}\\
e(\L)&=&\nu_1-\kappa^+-2\sum \alpha_j.\label{P4.4e2}
\end{eqnarray}
\end{proposition}

\begin{proof}
Case 1. $\nu_1=3$ and $\alpha_{\kappa^+}=1$. In this case we observe that $\L$ is the unknot if $\kappa^+=1$ and $\alpha_{\kappa^+}=1$, hence $E(\L)=e(\L)=0$ as given in (\ref{P4.4e1}) and (\ref{P4.4e2}). If $\kappa^+\ge 2$, then $\L$ is equivalent to a link in $P_3(2;-3\vert 2\alpha_1,\ldots, 2\alpha_{\kappa^+-1};0)$. Thus (\ref{P4.4e1}) and (\ref{P4.4e2}) follow from Proposition \ref{P4.1}.

\medskip
Case 2. $\nu_1\ge 3$, $\kappa^+=1$ and $\alpha_{\kappa^+}>1$. 
In this case $D$ reduces to an alternating link diagram $\tilde{D}$ as shown in Figure \ref{mu1=1_nu1}.  We have $s(\tilde{D})=s(D)-2=2\alpha_1-1$, $w(\tilde{D})=w(D)-2=2\alpha_1-1-\nu_1$, $r^-(\tilde{D})=0$, $r^+(\tilde{D})=\alpha_1-2$. By \cite[Theorem 4.7]{Diao2021}, we have 
\begin{eqnarray*}
E(\L)&=&s(\tilde{D})-w(\tilde{D})-1=\nu_1-1=\nu_1-\kappa^+, \\
e(\L)&=&-s(\tilde{D})-w(\tilde{D})+1+2r^+(\tilde{D})=\nu_1-1-2\alpha_1=\nu_1-\kappa^+-2\sum \alpha_j.
\end{eqnarray*}
So (\ref{P4.4e1}) and (\ref{P4.4e2}) hold.
\begin{figure}[htb!]
\includegraphics[scale=.7]{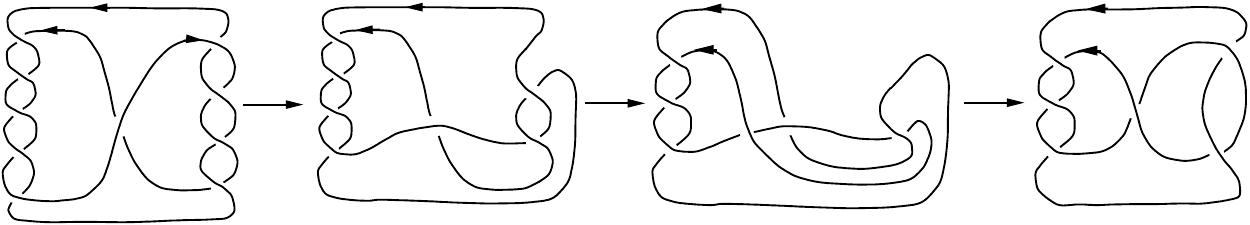}
\caption{$P_3(1;-\nu_1|2\alpha_1;0)$ reduces to the alternating link $P_3(0;-\nu_1-2,-1|2(\alpha_1-1);0)$.
}
\label{mu1=1_nu1} 
\end{figure}

\medskip
Case 3. $\kappa^+\ge 1$, $\nu_1= 3$ and $\alpha_{\kappa^+}>1$. 
We use induction on $\kappa^+$ and $\alpha_{\kappa^+}$. The initial step $\kappa^+=1$ is established in Case 2. Assuming that the statement is true for $1\le  \kappa^+\le q$ for some $q\ge 1$, consider the case $\kappa^+=q+1$ starting with $\alpha_{\kappa^+}=2$.
Apply VP$^+$ to a crossing in the $\alpha_{\kappa^+}$ strip. Case 1 applies to $D_-$ and the induction hypothesis applies to $D_0$. We have 
\begin{eqnarray*}
-1+E(D_0)&=&-1+\nu_1-q=\nu_1-\kappa^+,\\
-1+e(D_0)&=&-1+\nu_1-q-2\sum_{1\le j\le q} \alpha_j=4+\nu_1-\kappa^+-2\sum \alpha_j,\\
-2+E(D_-)&=&-2+\nu_1-\kappa^+-2=\nu_1-\kappa^+-4,\\
-2+e(D_-)&=&-2+\nu_1-\kappa^+-2-2\sum_{1\le j\le q} \alpha_j=\nu_1-\kappa^+-2\sum \alpha_j.
\end{eqnarray*}
The result follows by simple comparison of the above.

\medskip
Case 4. $\nu_1> 3$. Use induction on $\kappa^+$, starting from $\kappa^+=1$. If $\alpha_1=1$, then $\L$ is equivalent to the torus link $T_p(-(\nu_1-2),2)$, and (\ref{P4.4e1}), (\ref{P4.4e2}) follow by direct computation using (\ref{HTp(-n)}). If $\alpha_1>1$, then the result follows from Case 2. Assume now that (\ref{P4.4e1}), (\ref{P4.4e2}) hold for $1\le \kappa^+\le q$ for some $q\ge 1$ and consider $\kappa^+=q+1$. If $\alpha_{\kappa^+}=1$, then $\L$ is equivalent to a link in 
$P_3(2;-\nu_1\vert 2\alpha_1,\ldots, 2\alpha_{\kappa^+-1};0)$, and the result follows from case (c) of Proposition \ref{P4.2}. Assume (\ref{P4.4e1}), (\ref{P4.4e2}) hold for $1\le \alpha_{\kappa^+}\le z$ for some $z\ge 1$ and consider $\alpha_{\kappa^+}=z+1$. Apply VP$^+$ to a crossing in the $\alpha_{\kappa^+}$ strip. The induction hypothesis applies to both $D_0$ and $D_-$. Simple calculation and comparison then shows that $E(\L)=-1+E(D_0)=\nu_1-\kappa^+$ and $e(\L)=-2+e(D_-)=\nu_1-\kappa^+-2\sum\alpha_j$, as desired. This completes the proof of Proposition \ref{P4.4}.
\end{proof}

\medskip
\begin{proposition}\label{P4.5} Let $\L\in P_3(1;-\nu_1\vert 2\alpha_1,\ldots,2\alpha_{\kappa^+};-2\beta_1,\ldots,-2\beta_{\kappa^-})$ such that $\nu_1\ge 2$ and $\kappa^->0$, then
\begin{eqnarray}
E(\L)&=&\nu_1+\kappa^--\kappa^++2\sum \beta_i,\label{P4.5e1}\\
e(\L)&=&
\left\{
\begin{array}{ll}
\nu_1+\kappa^--\kappa^+-2, &{\rm if} \ \kappa^+=0,\\
\nu_1+\kappa^--\kappa^+-2\sum \alpha_j, &{\rm if} \ \kappa^+>0.
\end{array}
\right. \label{P4.5e2}
\end{eqnarray}
\end{proposition}

\begin{proof} We use induction on $\kappa^+$, starting from $\kappa^+=0$. In this case, an N-move as shown in Figure \ref{Type3fig} reduces $D$ to a Type M2 link diagram $\hat{D}$ as defined in \cite{Diao2021}. We have $s(\hat{D})=3-\kappa^-+2\sum \beta_i$, $w(\hat{D})=-\nu_1-2\sum \beta_i$, $r^+(\hat{D})=0$ and $r^-(\hat{D})=1-\kappa^-+\sum \beta_i$, and the result follows from  \cite[Theorem 4.7]{Diao2021} which asserts that $E(\L)=E(\hat{D})=s(\hat{D})-w(\hat{D})-1-2r^-(\hat{D})$ and $e(\L)=e(\hat{D})=-s(\hat{D})-w(\hat{D})+1+2r^+(\hat{D})$.

\medskip
Now consider the case $\kappa^+=1$, starting from $\alpha_1=1$. In this case $\L$ is equivalent to a link in $P_3(2;-\nu_1\vert 0;-2\beta_1,\ldots,-2\beta_{\kappa^-})$ and the result follows from Proposition \ref{P4.2}. If $\alpha_1=2$, apply VP$^+$ to a crossing in the $\alpha_1$ strip. By the above discussions for the case of $\kappa^+=0$ and $\kappa^+=1$, $\alpha_1=1$, we have
\begin{eqnarray*}
-1+E(D_0)&=&\nu_1+\kappa^--1+2\sum\beta_i=\nu_1+\kappa^--\kappa^++2\sum\beta_i,\\
-1+e(D_0)&=&-3+\nu_1+\kappa^-,\\
-2+E(D_-)&=&\nu_1+\kappa^--3+2\sum\beta_i,\\
-2+e(D_-)&=&-5+\nu_1+\kappa^-=\nu_1+\kappa^--\kappa^+-2\sum\alpha_j.
\end{eqnarray*}
The result follows trivially. RLR for the general case $\alpha_1\ge 3$.

\medskip
Assume that the statement of the proposition holds for $1\le \kappa^+\le q$ for some $q\ge 1$ and consider $\kappa^+=q+1$. If $\alpha_{\kappa^+}=1$, then again $\L$ is equivalent to a link in $P_3(2;-\nu_1\vert 2\alpha_1,\ldots,2\alpha_{\kappa^+-1};-2\beta_1,\ldots,-2\beta_{\kappa^-})$ and the result follows from Proposition \ref{P4.2}. RLR for the induction step (in which case we apply VP$^+$ to a crossing in the $\alpha_{\kappa^+}$ strip and the induction hypothesis applies to both $D_-$ and $D_0$). This completes the proof of Proposition \ref{P4.5}.
\end{proof}

\section{The cases when $\rho^++\rho^-\ge 4$, $\kappa^++\kappa^->0$ or $\rho^++\rho^-=2$, $\rho^+\cdot \rho^-=0$, $\kappa^++\kappa^->0$}\label{Type3S3}

\medskip
The remaining cases that are not covered in Sections \ref{Type3basic_case} and \ref{Type3S2} are $\rho^++\rho^-\ge 4$, $\kappa^++\kappa^->0$ or $\rho^++\rho^-=2$, $\rho^+\cdot \rho^-=0$, $\kappa^++\kappa^->0$. For the first case, keep in mind that we only need to consider the link diagrams under the condition that $\nu_i>1$ (if $\rho^->0$). For the second case, we only need to consider $\rho^+=2$, $\rho^-=0$ and $\kappa^++\kappa^->0$ as the case $\rho^+=0$, $\rho^-=2$ and $\kappa^++\kappa^->0$ can then be established using the mirror image argument. We first make the following observation.

\begin{remark}
\label{additional_reduction}
{\em If $\delta^+\ge 1$, then one of these lone crossings can be combined with a positive strip $2\alpha_j$ to create a new diagram with one less Seifert circles as shown in Figure \ref{alpha_sigma_reduction}. We shall call such a move an $A$-move. Notice that after an $A$-move, there are still $\alpha_j-1$ (local) reduction moves as before.}
\end{remark}

\begin{figure}[htb!]
\includegraphics[scale=.7]{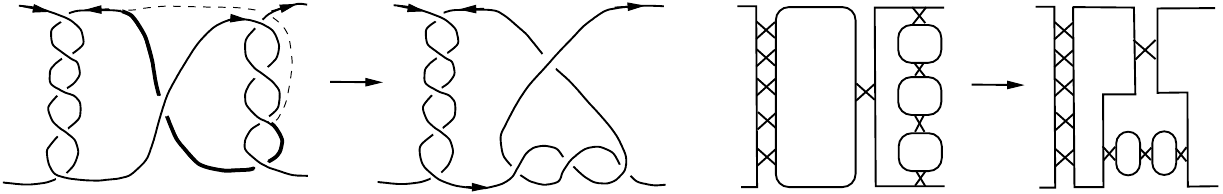}
\caption{Left: the illustration of an $A$-move, where the strand to be re-routed is highlighted by a thickened line and the diagram in a better drawing after the move; Right: the effect of the $A$-move on the Seifert circle decomposition of the diagram.} 
\label{alpha_sigma_reduction} 
\end{figure}

\begin{proposition}\label{P5.2}
Let $D$ be a standard diagram of $\L\in P_3(\mu_1,\ldots,\mu_{\rho^+};-\nu_1,\ldots,-\nu_{\rho^-}\vert 0; -2\beta_1,\ldots,-2\beta_{\kappa^-})$ 
with $2n=\rho^++\rho^-\ge 4$ and $\delta^+>\rho^-$, then we have
\begin{eqnarray}
E(\L)&=&s(D)-w(D)-1-2r^-(D),\ p_0^h(\L)\in  (-1)^{c^-(D)}F,\\
e(\L)&=&-s(D)-w(D)+1+2\min\{\delta^+,n-1\},\\
p^\ell_0(\L)&\in &
\left\{
\begin{array}{ll}
(-1)^{\rho^-+\kappa^-+c^-(D)}F,&\ {\rm if}\ \delta^+<n-1,\\
(-1)^{1+\rho^-+\kappa^-+c^-(D)}F,&\ {\rm if}\ \delta^+\ge n-1.
\end{array}
\right.
\end{eqnarray}
\end{proposition}

\begin{proof} Consider first the case $\beta_i=1$ for $1\le i\le \kappa^-$ (so $r^-(D)=0$). First change $D$ to a Type B diagram $\tilde{D}$ by an N-move as demonstrated in Figure \ref{Type3fig}. The structure of $\tilde{D}$ is almost identical to its counterpart in Proposition \ref{P3.5}, with the only difference being that
the diagram $D_1$ in $\hat{D}$ here has $\kappa^-$ additional Seifert circles attached to $C$, each via $2$ negative crossings. Thus the result follows from the same calculation as we did in Proposition \ref{P3.5}. (The only difference is that when applying Remark \ref{elementary_remark} to the $e$-power of $D_1$, $\kappa^-$ needs to be added to the powers of $(-1)$.)
This proves the case when $\beta_i=1$ for all $1\le i\le \kappa^-$. If we use induction on $\sum \beta_i$, this also establishes the first step $\sum \beta_i=1$. Assume that the statement is true for some $\sum \beta_i$ values the sum is increased by one. If after this increase we still have $\beta_i=1$ for all $1\le i\le \kappa^-$, then the statement holds and there is nothing more to prove. Thus we only need to consider the case that $\beta_i>1$ for some $i$. Say we have $\beta_{\kappa^-}\ge 2$. Apply VP$^-$ to a crossing in the $\beta_{\kappa^-}$ strip. The induction hypothesis now applies to both $D_0$ and $D_+$. We have $s(D_0)=s(D)+1-2\beta_{\kappa^-}$, $w(D_0)=w(D)+2\beta_{\kappa^-}$, $c^-(D_0)=c^-(D)-2\beta_{\kappa^-}$, $r^-(D_0)=r^-(D)+1-\beta_{\kappa^-}$, $s(D_+)=s(D)-2$, $w(D_+)=w(D)+2$, $c^-(D_+)=c^-(D)-2$, $r^-(D_+)=r^-(D)-1$, it follows that
\begin{eqnarray*}
1+E(D_0)&=&1+s(D_0)-w(D_0)-1-2r^-(D_0) \\
&=& s(D)-w(D)-1-2\beta_{\kappa^-}-2r^-(D),\\
1+e(D_0)&=&1-s(D_0)-w(D_0)+1+2\min\{\delta^+,n-1\}\\
&=& -s(D)-w(D)+1+2\min\{\delta^+,n-1\},\\
-zp_0^\ell(D_0)&\in & 
\left\{
\begin{array}{ll}
(-1)(-1)^{\rho^-+(\kappa^--1)+(c^-(D)-2\beta_{\kappa^-})}F,&\ {\rm if}\ \delta^+<n-1,\\
(-1)(-1)^{1+\rho^-+(\kappa^--1)+(c^-(D)-2\beta_{\kappa^-})}F,&\ {\rm if}\ \delta^+\ge n-1,
\end{array}
\right.\\
&=& 
\left\{
\begin{array}{ll}
(-1)^{\rho^-+\kappa^-+c^-(D)}F,&\ {\rm if}\ \delta^+<n-1,\\
(-1)^{1+\rho^-+\kappa^-+c^-(D)}F,&\ {\rm if}\ \delta^+\ge n-1,
\end{array}
\right.\\
2+E(D_+)&=&2+(s(D)-2)-(w(D)+2)-1-2(r^-(D)-1)\\
&=&s(D)-w(D)-1-2r^-(D),\\
p_0^h(D_+)&\in& (-1)^{c^-(D)-2}F=(-1)^{c^-(D)}F,\\
2+e(D_-)&=&2-(s(D)-2)-(w(D)+2)+1+2\min\{\delta^+,n-1\}\\
&=&-s(D)-w(D)+3+2\min\{\delta^+,n-1\}.
\end{eqnarray*}
By simple comparison, we see that the statement holds and the induction is complete.
\end{proof}

\medskip
\begin{corollary}\label{C5.3}
The statements of Proposition \ref{P5.2} hold if $\rho^+=2$, $\rho^-=0$ and $\kappa^+=0$. That is, 
\begin{eqnarray}
E(\L)&=&s(D)-w(D)-1-2r^-(D),\ p_0^h(\L)\in  F,\\
e(\L)&=&-s(D)-w(D)+1,\ p^\ell_0(\L)\in (-1)^{1+\kappa^-}F
\end{eqnarray}
since $\rho^-=0$ and $c^-(D)$ is even.
\end{corollary} 

\begin{proof}
One can verify that the proof of Proposition \ref{P5.2} goes through by substituting $\rho^-=0$ and ignoring the restriction on $\delta^+$ (since there are no N-moves can be made). Alternatively, one can prove this by following the proof of \cite[Theorem 4.7]{Diao2021} since $\L$ in this case is a Type B alternating link as defined in \cite{Diao2021}.
\end{proof}

\medskip
\begin{proposition}\label{P5.4}
Let $\L\in P_3(\mu_1,\ldots,\mu_{\rho^+};-\nu_1,\ldots,-\nu_{\rho^-}\vert 2\alpha_1,\ldots, 2\alpha_{\kappa^+}; -2\beta_1,\ldots, -2\beta_{\kappa^-})$ and $D$ be a standard diagram of $\L$. If $2n=\rho^++\rho^-\ge 4$ and $\delta^+> \rho^-+\kappa^+$, then we have
\begin{eqnarray}
E(\L)&=&s(D)-w(D)-1-2r^-(D),\\
p_0^h(\L)&\in& (-1)^{c^-(D)}F,\\
e(\L)&=&-s(D)-w(D)+1+2\kappa^++2\min\{\delta^+-\kappa^+,n-1\}+2r^+(D),\\
p_0^\ell(\L)&\in&
\left\{
\begin{array}{ll}
(-1)^{\rho^-+\kappa^-+c^-(D)}F,&\ {\rm if}\ \delta^+-\kappa^+<n-1,\\
(-1)^{1+\rho^-+\kappa^-+c^-(D)}F,&\ {\rm if}\ \delta^+-\kappa^+\ge n-1.
\end{array}
\right.
\end{eqnarray}
\end{proposition}

\begin{proof} Let us first consider the case when $\alpha_j=1$ for all $j$. We shall first perform $\kappa^+$ $A$-moves as shown in Figure \ref{alpha_sigma_reduction}. Let us denote the resulting diagram by $\mathbb{D}$. Notice that $\mathbb{D}$ is now a diagram for a link in
$P_3(\mu_1,\ldots,\mu_{\rho^+-\delta^+},2,\ldots,2,1,\ldots,1;-\nu_1,\ldots,-\nu_{\rho^-}\vert 0;-2\beta_1,\ldots,-2\beta_{\kappa^-})$ (with $\kappa^+$ 2's and $\delta^+-\kappa^+$ 1's). At this point, $\mathbb{D}$ has the same cycle length as $D$ has, but it has $\kappa^+$ less lone crossings, with each lone crossing replaced by a vertical strip of 2 crossings. Also, $\mathbb{D}$ has $\delta^+-\kappa^+> \rho^-$ lone crossings, hence Proposition \ref{P5.2} applies to $\mathbb{D}$. We have $s(\mathbb{D})=s(D)-\kappa^+$, $w(\mathbb{D})=w(D)-\kappa^+$, $c(\mathbb{D})=c(D)-\kappa^+$, $c^-(\mathbb{D})=c^-(D)$, $r^-(\mathbb{D})=r^-(D)$, $r^+(\mathbb{D})=0$. Thus we have
\begin{eqnarray*}
E(\L)&=&E(\mathbb{D})=s(\mathbb{D})-w(\mathbb{D})-1-2r^-(\mathbb{D})\\
&=&s(D)-w(D)-1-2r^-(D),\\
p_0^h(\L)&=&p_0^h(\mathbb{D})\in  (-1)^{c^-(\mathbb{D})}F=(-1)^{c^-(D)}F,\\
e(\L)&=&e(\mathbb{D})=-s(\mathbb{D})-w(\mathbb{D})+1+2\min\{\delta^+-\kappa^+,n-1\}\\
&=&-s(\mathbb{D})-w(\mathbb{D})+1+2\kappa^++2\min\{\delta^+-\kappa^+,n-1\},\\
p^\ell_0(\L)=p_0^\ell(\mathbb{D})&\in&
\left\{
\begin{array}{ll}
(-1)^{\rho^-+\kappa^-+c^-(\mathbb{D})}F,&\ {\rm if}\ \delta^+-\kappa^+<n-1,\\
(-1)^{1+\rho^-+\kappa^-+c^-(\mathbb{D})}F,&\ {\rm if}\ \delta^+-\kappa^+\ge n-1,
\end{array}
\right.\\
&=&
\left\{
\begin{array}{ll}
(-1)^{\rho^-+\kappa^-+c^-(D)}F,&\ {\rm if}\ \delta^+-\kappa^+<n-1,\\
(-1)^{1+\rho^-+\kappa^-+c^-(D)}F,&\ {\rm if}\ \delta^+-\kappa^+\ge n-1.
\end{array}
\right.
\end{eqnarray*}
Thus the statement holds. Use induction on $\sum \alpha_j\ge \kappa^+$, the above has established the initial step $\sum \alpha_j= \kappa^+$. Assume that the statement holds for $\sum \alpha_j\ge \kappa^+$ up to certain values (greater than or equal to $\kappa^+$) and consider the case that it is increased by one. It is necessary that one of the $\alpha_j$'s is now greater than 1, say $\alpha_1\ge 2$. Apply VP$^+$ to a crossing in the $\alpha_1$ strip. The induction hypothesis  applies to both $D_-$ and $D_0$ with $D_0$ having one less $\alpha_j$ strips. We have
$s(D_0)=s(D)+1-2\alpha_1$, $w(D_0)=w(D)-2\alpha_1$, $r^+(D_0)=r^+(D)+1-\alpha_1$, $ r^-(D_0)=r^-(D)$, $c^-(D_0)=c^-(D)$, $s(D_-)=s(D)-2$, $w(D_-)=w(D)-2$, $r^+(D_-)=r^+(D)-1$, $c^-(D_-)=c^-(D)$. It follows that
\begin{eqnarray*}
-1+E(D_0)&=&-1+(s(D)+1-2\alpha_1)-(w(D)-2\alpha_1)-1-2r^-(D)\\
&=&s(D)-w(D)-1-2r^-(D),\\
zp_0^h(D_0)&\in &(-1)^{c^-(D_0)}F=(-1)^{c^-(D)}F,\\
-2+E(D_-)&=&-2+(s(D)-2)-(w(D)-2)-1-2r^-(D)\\
&=&s(D)-w(D)-3-2r^-(D),\\
-1+e(D_0)&=&-1-(s(D)+1-2\alpha_1)-(w(D)-2\alpha_1)\\
&+&1+2(\kappa^+-1)+2\min\{\delta^++1-\kappa^+,n-1\}+2(r^+(D)+1-\alpha_1)\\
&=&-s(D)-w(D)-1+2\alpha_1+2\kappa^++2\min\{\delta^+-(\kappa^+-1),n-1\}+2r^+(D)\\
&\ge & -s(D)-w(D)+3+2\kappa^++2\min\{\delta^+-\kappa^+,n-1\}+2r^+(D),\\
-2+e(D_-)&=&-2-(s(D)-2)-(w(D)-2)+1+2\kappa^++2\min\{\delta^+-\kappa^+,n-1\}+2(r^+(D)-1)\\
&=&-s(D)-w(D)+1+2\kappa^++2\min\{\delta^+-\kappa^+,n-1\}+2r^+(D),\\
p_0^\ell(\L)=p_0^\ell(D_-)&\in &
\left\{
\begin{array}{ll}
(-1)^{\rho^-+\kappa^-+c^-(D)}F,&\ {\rm if}\ \delta^+-\kappa^+<n-1,\\
(-1)^{1+\rho^-+\kappa^-+c^-(D)}F,&\ {\rm if}\ \delta^+-\kappa^+\ge n-1.
\end{array}
\right.
\end{eqnarray*}
The statement follows by comparison and the theorem is proved.
\end{proof}

\medskip
\begin{corollary}\label{C5.5}
The statements of Proposition \ref{P5.4} hold if $\rho^+=\delta^+=2$, $\rho^-=0$ and $\kappa^+=1$. That is, 
\begin{eqnarray}
&E(\L)=s(D)-w(D)-1-2r^-(D),&e(\L)=-s(D)-w(D)+1+2\alpha_1.
\end{eqnarray}
\end{corollary} 

\begin{proof}
One can verify that the proof of Proposition \ref{P5.4} goes through by substituting $\rho^-=0$, $\rho^+=\delta^+=2$ and $\kappa^+=1$. \end{proof}

\medskip
\begin{proposition}\label{P5.6}
Let  $\L\in P_3(\mu_1,\ldots,\mu_{\rho^+};-\nu_1,\ldots,-\nu_{\rho^-}\vert 2\alpha_1,\ldots, 2\alpha_{\kappa^+}; -2\beta_1,\ldots, -2\beta_{\kappa^-})$ with $\rho^++\rho^-\ge 4$ and $D$ be a standard diagram of $\L$ such that $\delta^+\le \rho^-+\kappa^+$.  Then we have
\begin{eqnarray}
E(\L)&=&s(D)-w(D)-1-2r^-(D),\ p_0^h(\L)\in (-1)^{c^-(D)}F,\label{genEpower}\\
e(\L)&=&-s(D)-w(D)+1+2\delta^++2r^+(D),\ p_0^\ell(\L)\in(-1)^{1+\kappa^++\kappa^-+\delta^++c^-(D)}F. \label{genepower}
\end{eqnarray}
\end{proposition}

\begin{proof}
We shall use induction on $\kappa=\kappa^++\kappa^-$. The statement holds for $\kappa=0$ by Proposition \ref{P3.4}. 
Assume that the statement holds for some $\kappa\ge 0$, we need to consider the case when it is increased by one. Consider first the case that $\kappa^-$ is increased by one, that is, $\kappa^-\ge 1$ and the statement holds for $\kappa^--1$ and the current $\kappa^+$, then we start with $\beta_{\kappa^-}=1$ and apply $VP^-$ on a crossing in the $\beta_{\kappa^-}$ strip. The induction hypothesis applies to $D_0$ since the condition $\delta^+\le \rho^-+\kappa^+$ still holds, and $D_+$ is the connected sum of the torus links $T_p(\mu_j,2)$, $1\le j\le \rho^+-\delta^+$, $T_p(-\nu_j,2)$, $1\le j\le \rho^-$, $T_o(2\alpha_j,2)$, $1\le j\le \kappa^+$, and $T_o(-2\beta_i,2)$, $1\le j\le \kappa^--1$. We have $s(D_0)=s(D)-1$, $w(D_0)=w(D)+2$, $c^-(D_0)=c^-(D)-2$, $r^+(D_0)=r^+(D)$, $r^-(D_0)=r^-(D)$, $s(D_+)=s(D)-\delta^+$, $w(D_+)=w(D)+2-\delta^+$, $r^+(D_+)=r^+(D)$ and $r^-(D_+)=r^-(D)$. 
It follows that (by using (\ref{TorusConnect3}) and  (\ref{TorusConnect4})  in  Remark \ref{toruslink_formula} for $D_+$)
\begin{eqnarray*}
1+E(D_0)&=&1+s(D_0)-w(D_0)-1-2r^-(D_0) = s(D)-w(D)-3-2r^-(D),\\
1+e(D_0)&=&1-s(D_0)-w(D_0)+1+2\delta^++2r^+(D_0)\\
& =& -s(D)-w(D)+1+2\delta^++2r^+(D),\\
-zp_0^\ell(D_0)&\in& (-1)(-1)^{1+\kappa^++(\kappa^--1)+\delta^++c^-(D)+2}F\\
&=&(-1)^{1+\kappa^++\kappa^-+\delta^++c^-(D)}F,\\
2+E(D_+)&=&2+(s(D)-\delta^+)-(w(D)+2-\delta^+)-1-2r^-(D)\\
&=&s(D)-w(D)-1-2r^-(D),\\
p_0^h(D_+)&\in &(-1)^{\sum \nu_i}F=(-1)^{c^-(D)}F,\\
2+e(D_+)&=&2-(s(D)-\delta^+)-(w(D)+2-\delta^+)+1+2r^+(D)\\
&=&-s(D)-w(D)+1+2\delta^++2r^+(D),\\
p_0^\ell(D_+)&\in &(-1)^{\rho^+-\delta^++\rho^-+\sum \nu_i+\kappa^++\kappa^--1}F\\
&=&(-1)^{1+\kappa^++\kappa^-+\delta^++c^-(D)}F.
\end{eqnarray*}
We see that the statement of the proposition holds by comparison of the above results. Assume now that the statement holds for $\beta_{\kappa^-}-1\ge 1$, then for $\beta_{\kappa^-}$, again apply $VP^-$ on a crossing in the $\beta_{\kappa^-}$ strip. The induction hypothesis applies to both $D_0$ and $D_+$. We have $s(D_0)=s(D)+1-2\beta_{\kappa^-}$, $w(D_0)=w(D)+2\beta_{\kappa^-}$, $c^-(D_0)=c^-(D)+2\beta_{\kappa^-}$, $r^+(D_0)=r^+(D)$, $r^-(D_0)=r^-(D)+1-\beta_{\kappa^-}$, $s(D_+)=s(D)-2$, $w(D_+)=w(D)+2$, $c^-(D_+)=c^-(D)+2\beta_{\kappa^-}$, $r^+(D_+)=r^+(D)$ and $r^-(D_-)=r^-(D)-1$. It follows that
\begin{eqnarray*}
1+E(D_0)&=&1+s(D_0)-w(D_0)-1-2r^-(D_0) = s(D)-w(D)-1-2\beta_{\kappa^-}-2r^-(D),\\
2+E(D_+)&=&2+(s(D)-2)-(w(D)+2)-1-2(r^-(D)-1)=s(D)-w(D)-1-2r^-(D),\\
p_0^h(D_+)&\in& (-1)^{c^-(D)+2\beta_{\kappa^-}}F=(-1)^{c^-(D)}F,\\
1+e(D_0)&=&1-s(D_0)-w(D_0)+1+2\delta^++2r^+(D_0)\\
& =& -s(D)-w(D)+1+2\delta^++2r^+(D),\\
-zp_0^\ell(D_0)&\in & (-1)(-1)^{1+\kappa^++(\kappa^--1)+\delta^++c^-(D)+2\beta_{\kappa^-}}F\\
&=&(-1)^{1+\kappa^++\kappa^-+\delta^++c^-(D)}F,\\
2+e(D_-)&=&2-(s(D)-2)-(w(D)+2)+1+2\delta^++2r^+(D)\\
&=&-s(D)-w(D)+3+2\delta^++2r^+(D).
\end{eqnarray*}
Comparison now shows that the statement of the theorem still holds.

\medskip
Let us now consider the case when $\kappa^+$ is increased by one. That is, $\kappa^+\ge 1$ and the statement holds for $\kappa^+-1$ and $\kappa^-$. Notice that it means the statement holds under the condition $\delta^+\le \rho^-+\kappa^+-1$. If $\delta^+\le \rho^-+\kappa^+-1$, then we start with $\alpha_{\kappa^+}=1$ and apply $VP^+$ on a crossing in the $\alpha_{\kappa^+}$ strip. The induction hypothesis applies to $D_0$. We have $s(D_0)=s(D)-1$, $w(D_0)=w(D)-2$, $c(D_0)=c(D)-2$, $c^-(D_0)=c^-(D)$, $r^+(D_0)=r^+(D)$ and $r^-(D_0)=r^-(D)$. It follows that
\begin{eqnarray*}
-1+E(D_0)&=&-1+s(D_0)-w(D_0)-1-2r^-(D_0) = s(D)-w(D)-1-2r^-(D),\\
z p_0^h(D_0)&\in &
(-1)^{c^-(D_0)}F=(-1)^{c^-(D)}F,\\
-1+e(D_0)&=&-1-s(D_0)-w(D_0)+1+2\delta^++2r^+(D_0)\\
& =& -s(D)-w(D)+3+2\delta^++2r^+(D).
\end{eqnarray*}
On the other hand,
$D_-$ is the connected sum of the torus links $T_p(\mu_j,2)$, $1\le j\le \rho^+-\delta^+$, $T_p(-\nu_j,2)$, $1\le j\le \rho^-$, $T_o(2\alpha_j,2)$, $1\le j\le \kappa^+-1$, and $T_o(-2\beta_i,2)$, $1\le j\le \kappa^-$. We have $s(D_-)=s(D)-\delta^+$, $w(D_-)=w(D)-2-\delta^+$, $r^+(D_-)=r^+(D)$ and $r^-(D_-)=r^-(D)$. Direct calculations using (\ref{TorusConnect3}) and  (\ref{TorusConnect4}) in Remark \ref{toruslink_formula} then lead to
\begin{eqnarray*}
-2+E(D_-)&=&-2+(s(D)-\delta^+)-(w(D)-2-\delta^+)-1-2r^-(D)\\
&=&s(D)-w(D)-1-2r^-(D),\\
p_0^h(D_-)&\in &(-1)^{\sum \nu_i}F=(-1)^{c^-(D)}F,\\
-2+e(D_-)&=&-2-(s(D)-\delta^+)-(w(D)-2-\delta^+)+1+2r^+(D)\\
&=&-s(D)-w(D)+1+2\delta^++2r^+(D),\\
p_0^\ell(D_-)&\in &(-1)^{\rho^+-\delta^++\rho^-+\sum \nu_i+(\kappa^+-1)+\kappa^-}F\\
&=&(-1)^{1+\kappa^++\kappa^-+\delta^++c^-(D)}F.
\end{eqnarray*}
We see that the statement of the theorem holds by comparison of the above results. Assume now that the statement holds for $\alpha_{\kappa^+}-1\ge 1$, then for $\alpha_{\kappa^+}$, again apply $VP^+$ on a crossing in the $\alpha_{\kappa^+}$ strip. The induction hypothesis applies to both $D_0$ and $D_-$. We have $s(D_0)=s(D)+1-2\alpha_{\kappa^+}$, $w(D_0)=w(D)-2\alpha_{\kappa^+}$, $c^-(D_0)=c^-(D)$, $r^+(D_0)=r^+(D)+1-\alpha_{\kappa^+}$ and $r^-(D_0)=r^-(D)$. It follows that
\begin{eqnarray*}
-1+E(D_0)&=&-1+s(D_0)-w(D_0)-1-2r^-(D_0) = s(D)-w(D)-1-2r^-(D),\\
z p_0^h(D_0)&\in &
(-1)^{c^-(D_0)}F=(-1)^{c^-(D)}F,\\
-1+e(D_0)&=&-1-s(D_0)-w(D_0)+1+2\delta^++2r^+(D_0)\\
& =& -s(D)-w(D)+1+2\alpha_{\kappa^+}+2\delta^++2r^+(D).
\end{eqnarray*}
On the other hand, we have $s(D_-)=s(D)-2$, $w(D_-)=w(D)-2$, $c^-(D_-)=c^-(D)$, $r^+(D_-)=r^+(D)-1$ and $r^-(D_-)=r^-(D)$. It follows that
\begin{eqnarray*}
-2+E(D_-)&=&-2+(s(D)-2)-(w(D)-2)-1-2r^-(D)=s(D)-w(D)-3-2r^-(D),\\
-2+e(D_-)&=&-2-(s(D)-2)-(w(D)-2)+1+2\delta^++2(r^+(D)-1)=\\
&=&-s(D)-w(D)+1+2\delta^++2r^+(D),\\
p_0^\ell(D_-)&=&(-1)^{1+\kappa^++\kappa^-+\delta^++c^-(D)}F.
\end{eqnarray*}
Comparison now shows that the statement of the theorem still holds. 

\medskip
We still need to consider the case $\delta^+=\rho^-+\kappa^+$, since in this case the induction hypothesis does not apply to $D_0$ in the above argument. Here we shall use induction on $\alpha_{\kappa^+}$, starting with $\alpha_{\kappa^+}=1$. We first apply an $A$-move to $D$ and observe that now the induction hypothesis applies to the resulting diagram $\mathbb{D}$, since $\mathbb{D}$ now has $\delta^+-1$ lone crossings in its main cycle and has $\kappa^+-1$ $\alpha_j$ strips. We have $s(\mathbb{D})=s(D)-1$, $w(\mathbb{D})=w(D)-1$, $c^-(\mathbb{D})=c^-(D)$, $r^-(\mathbb{D})=r^-(D)$ and $r^+(\mathbb{D})=r^+(D)$, it follows that
\begin{eqnarray*}
E(\L)&=&E(\mathbb{D})=s(\mathbb{D})-w(\mathbb{D})-1-2r^-(\mathbb{D})\\
&=&s(D)-w(D)-1-2r^-(D),\\
p^h_0(\L)&=&p^h_0(\mathbb{D})\in (-1)^{c^-(\mathbb{D})}F=(-1)^{c^-(D)}F,\\
e(\L)&=&e(\mathbb{D})=-s(\mathbb{D})-w(\mathbb{D})+1+2(\delta^+-1)+2r^+(\mathbb{D})\\
&=&-s(D)-w(D)+1+2\delta^++2r^+(D),\\
p_0^\ell(\L)&=&p_0^\ell(\mathbb{D})\in (-1)^{1+(\kappa^+-1)+\kappa^-+(\delta^+-1)+c^-(\mathbb{D})}F\\
&=& (-1)^{1+\kappa^++\kappa^-+\delta^++c^-(D)}F.
\end{eqnarray*}
This proves the case $\alpha_{\kappa^+}=1$. Assume now that the statement is true for some $\alpha_{\kappa^+}$ value ($\ge 1$) and $\alpha_{\kappa^+}$ is increased by one, that is, the statement holds for $\alpha_{\kappa^+}-1$. Apply VP$^+$ to a crossing in the $\alpha_{\kappa^+}$ strip. The induction hypothesis applies to $D_-$ and Proposition \ref{P5.4} applies to $D_0$ since it has $\delta^+$ lone crossings in its main cycle and only $\kappa^+-1$ $\alpha_j$ strips (thus $\delta^+=\rho^-+\kappa^+>\rho^-+(\kappa^+-1)$). We have $s(D_0)=s(D)+1-2\alpha_{\kappa^+}$, $w(D_0)=w(D)-2\alpha_{\kappa^+}$, $c^-(D_0)=c^-(D)$, $r^-(D_0)=r^-(D)$, $r^+(D_0)=r^+(D)+1-\alpha_{\kappa^+}$, $s(D_-)=s(D)-2$, $w(D_-)=w(D)-2$, $r^-(D_-)=r^-(D)$, $r^+(D_-)=r^+(D)-1$, and $-1+2\alpha_{\kappa^+}\ge 3$ since $\alpha_{\kappa^+}\ge 2$. Also, since $\rho^+\ge \delta^+=\rho^-+\kappa^+$, we have $2n\ge 2\rho^-+\kappa^+$. It follows that $\rho^-\le n-1$ hence
$$
\kappa^++\min\{\delta^++1-\kappa^+,n-1\}=\kappa^++\min\{\rho^-+1,n-1\}\ge \kappa^++\rho^-=\delta^+.
$$
Thus we have
\begin{eqnarray*}
-1+E(D_0)&=&-1+s(D_0)-w(D_0)-1-2r^-(D_0) = s(D)-w(D)-1-2r^-(D),\\
z p_0^h(D_0)&\in &
(-1)^{c^-(D_0)}F=(-1)^{c^-(D)}F,\\
-1+e(D_0)&=&-1-s(D_0)-w(D_0)+1+2(\kappa^+-1)+2\min\{\delta^+-(\kappa^+-1),n-1\}+2r^+(D_0)\\
& =& -s(D)-w(D)-1+2\alpha_{\kappa^+}+2\kappa^++2\min\{\delta^++1-\kappa^+,n-1\}+2r^+(D)\\
&\ge & -s(D)-w(D)+3+2\delta^++2r^+(D),\\
-2+E(D_-)&=&-2+(s(D)-3)-(w(D)-3)-1-2r^-(D)=s(D)-w(D)-3-2r^-(D),\\
-2+e(D_-)&=&-2-(s(D)-2)-(w(D)-2)+1+2\delta^++2(r^+(D)-1)=\\
&=&-s(D)-w(D)+1+2\delta^++2r^+(D),\\
p_0^\ell(D_-)&\in&(-1)^{1+\kappa^++\kappa^-+\delta^++c^-(D)}F.
\end{eqnarray*}
Comparison now shows that the statement of the theorem still holds. 
\end{proof}

\medskip
\begin{corollary}\label{C5.7}
The statements of Proposition \ref{P5.6} hold if $\rho^+=2$, $\rho^-=0$ and $\delta^+\le \kappa^+$. That is, 
\begin{eqnarray}
&E(\L)=s(D)-w(D)-1-2r^-(D),&e(\L)=-s(D)-w(D)+1+2\delta^++2r^+(D).
\end{eqnarray}
\end{corollary} 

\begin{proof}
One can verify that the proof of Proposition \ref{P5.6} goes through by substituting $\rho^-=0$. The only exception is that at the initial induction step $\kappa=\kappa^++\kappa^-=0$, the statement of the proposition is guaranteed by Proposition \ref{P3.1} instead of Proposition \ref{P3.4}.
\end{proof}

\medskip
\begin{theorem}\label{Type3LowboundTheorem}
The formulas given in Theorems \ref{MT1} and \ref{MT2} are lower bounds of the braid indices of the corresponding pretzel links.
\end{theorem}

\begin{proof}
The results in Sections \ref{Type3basic_case}, \ref{Type3S2} and \ref{Type3S3} allow us to compute $\b_0(\L)=(E(\L)-e(\L))/2+1$ for any Type 3 pretzel link $\L$, since every Type 3 pretzel link has been discussed in at least one of the propositions in these sections. Thus one just needs to verify in each case that $\b_0(\L)$ matches with the expression on the right side of the corresponding formula in Theorems \ref{MT1} and \ref{MT2}. We list the formulas and the corresponding results needed for their proofs and leave the calculations of $\b_0(\L)$ to the reader. \end{proof}

\medskip
\begin{supertabular}{ll}
Formula (\ref{MT1e0}):& Proposition \ref{P3.3}\\
Formula (\ref{MT1e1}):& Proposition \ref{P4.4}, Remark \ref{R4.3}/\cite[Theorem 1.4, (7)]{Diao2024}\\
Formula (\ref{MT1e2}):& Propositions \ref{P4.1}, \ref{P4.4}, Remark \ref{R4.3}/\cite[Theorem 1.4, (8)]{Diao2024} \\
Formula (\ref{MT1e3}):& Propositions \ref{P4.1}, \ref{P4.2}, \ref{P4.5}, Corollary \ref{C5.3}\\
Formula (\ref{MT1e4}):& Propositions \ref{P4.1}, \ref{P4.4}\\
Formula (\ref{MT1e5}):& Corollary \ref{C5.7}\\
Formula (\ref{MT1e6}):& Remark \ref{R4.3}/\cite[Theorem 1.4, (8), (9)]{Diao2024}, Proposition \ref{P4.5}, 
Corollaries \ref{C5.5}, \ref{C5.7}\\
Formula (\ref{MT1e7}):& Proposition \ref{P4.2}, Corollary \ref{C5.7}\\
Formula (\ref{MT2e1}):& Propositions \ref{P5.4} \\
Formula (\ref{MT2e2}):& Proposition \ref{P5.6}\\
\end{supertabular}

\medskip
\section{The determination of braid index upper bounds}\label{Upper_Sec}

\begin{theorem}\label{Type3UpboundTheorem}
With the exception of formula (\ref{MT1e4}), the formulas given in Theorems \ref{MT1} and \ref{MT2} are upper bounds of the braid indices of the corresponding pretzel links, that is, we have $\b(\L)\le \b_0(\L)$ for each $\L$ covered in formulas (\ref{MT1e1}) to (\ref{MT1e3}), and (\ref{MT1e5}) to (\ref{MT2e2}). In the case of formula (\ref{MT1e4}), we have $\b(\L)\le \b_0(\L)+1$.
\end{theorem}

\begin{proof}
Let $\L\in P_3(\mu_1,\ldots,\mu_{\rho^+};-\nu_1,\ldots,-\nu_{\rho^-}\vert 2\alpha_1,\ldots, 2\alpha_{\kappa^+};-2\beta_1,\ldots,-2\beta_{\kappa^-})$ and $D$ a standard diagram of $\L$. Let $\rho^++\rho^-=2n$ be the length of the main cycle of Seifert circles in $D$. Keep in mind that we only need to prove the theorem for the cases where $\nu_i>1$ for $1\le i\le \rho^-$ when $\rho^->0$. For such Type 3 pretzel links, we can always perform $\min\{\delta^+,n-1\}-\kappa^++\sum \alpha_j$ MP reduction moves involving positive lone crossings and $-\kappa^-+\sum \beta_i$ MP reduction moves involving negative lone crossings. It follows that $2n-\min\{\delta^+,n-1\}+\sum \alpha_j+\sum \beta_i$ is a braid index upper bound of $\L$ since $s(D)=2n-\kappa^++2\sum \alpha_j-\kappa^-+2\sum \beta_i$. However, in many cases, this upper bound needs to be improved since it is larger than $\b_0(\L)$. In the following proofs, we shall identify the extra Seifert circle reduction moves so that the resulting diagram will have $\b_0(\L)$ Seifert circles.
We shall go through the formula list in Theorems \ref{MT1} and \ref{MT2}.

\medskip\noindent
Formula (\ref{MT1e0}): This is obvious.

\medskip\noindent
Formula (\ref{MT1e1}): Case (i) $\kappa^-=0, \kappa^+=1, \mu_1=1, \nu_1=3, \alpha_1=1$. In this case $\L$ is the unknot hence $\b(\L)=1=\alpha_1$.

\medskip
Case (ii) $\kappa^-=0, \kappa^+>1, \mu_1=1, \nu_1=3, \alpha_{\kappa^+}=1, \alpha_{\kappa^+-1}>2$. $\L\in P_3(2;-3\vert 2\alpha_1,\ldots,2\alpha_{\kappa^+-1};0)$ in this case. Thus the result follows from Case (i) of Formula (\ref{MT1e2}).

\medskip
Case (iii) $\kappa^-=0, \kappa^+>1, \mu_1=1, \nu_1=2, \alpha_{\kappa^+}>1$. In this case $\L\in P_2(2\alpha_1,\ldots,2\alpha_{\kappa^+};-2)$ and the result follows from \cite[Theorem 1.4, formula (7)]{Diao2024}.

\medskip\noindent
Formula (\ref{MT1e2}): Case (i) $\kappa^-=0,  \kappa^+>0, \mu_1>1, \nu_1=\mu_1+1, \alpha_{\kappa^+}>\mu_1$. 
Figure \ref{MT1e2i_fig} illustrates how the upper bound $1+\sum \alpha_j$ is achieved using $\L\in P_3(3;-4\vert 8,8,8;0)$ as an example. Notice that we can always place the $\mu_1$ strip at the left side of the diagram, and in general the $\mu_1$ and $\nu_1$ strips may be separated by $\alpha_j$ strips. The sequence of moves shown from picture 2 to 5 in Figure \ref{MT1e2i_fig} creates a long Seifert circle and reduces the number of Seifert circles by one. Two sets of compatible special moves are shown in Figure \ref{MT1e2i_fig} and the last set of special moves are left to the reader. 

\begin{figure}[htb!]
\includegraphics[scale=.55]{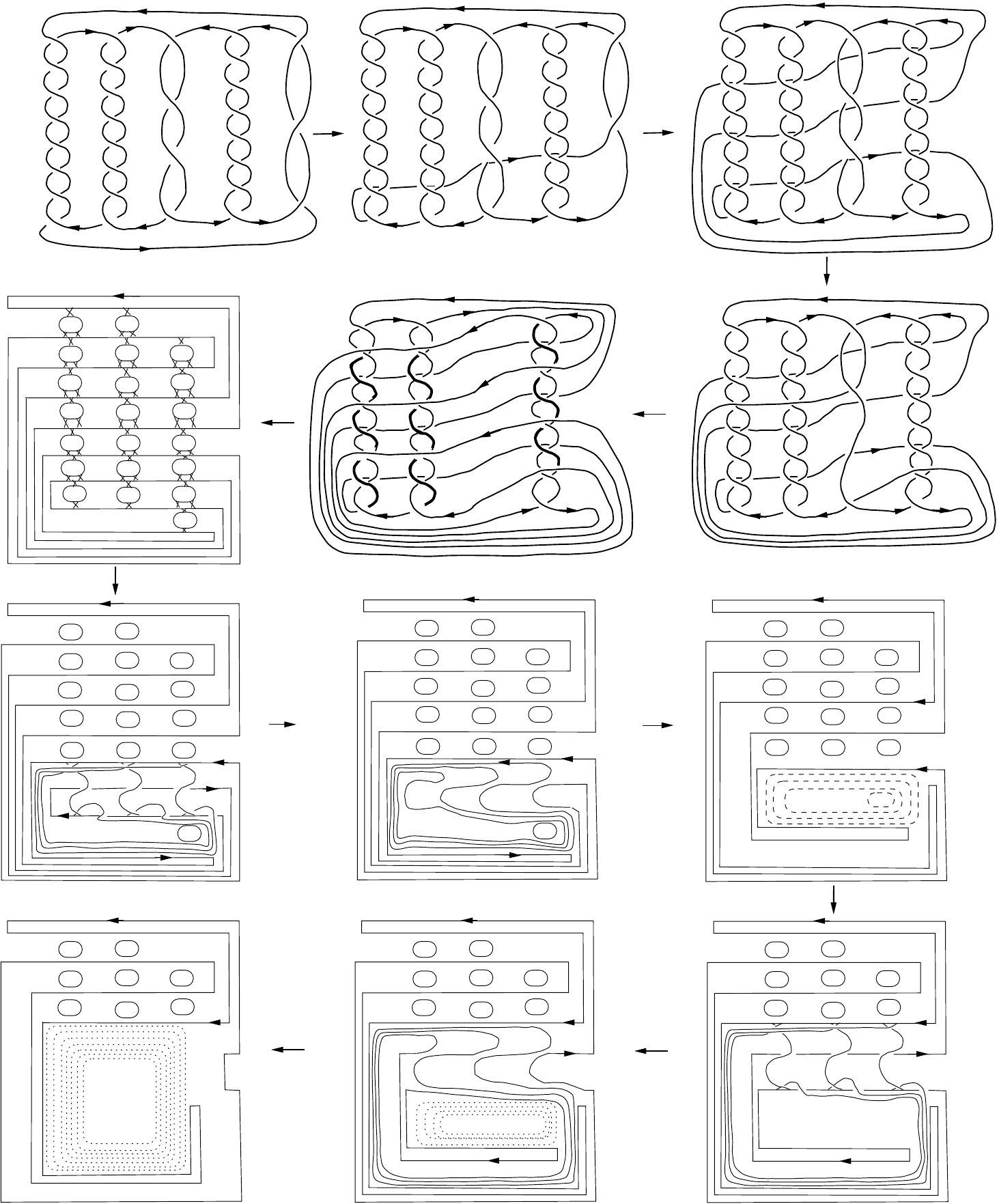}
\caption{The over-strands where the special moves can be made are marked by thickened curves in the fifth diagram. The sixth diagram shows the Seifert circle decomposition before the special moves are made. The first three compatible special moves are shown in the seventh diagram. Crossings between Seifert circles are mostly omitted since they do not affect the Seifert circle decomposition when the strands are re-routed in the same directions of the local braids.  } 
\label{MT1e2i_fig} 
\end{figure}

\medskip
Notice that the number of special moves we can make equals the number of MP-moves we can make on the original diagram, which equals $-\kappa^++\sum \alpha_j$. Thus $\L$ can be represented by a diagram with $1-\kappa^++2\sum \alpha_j-(-\kappa^++\sum \alpha_j)=1+\sum \alpha_j$ Seifert circles. Notice that this procedure can always be carried out. In general, if $D$ is a standard diagram of $\L\in P_3(\mu_1;-(\mu_1+1)\vert 2\alpha_1,\ldots,2\alpha_{\kappa^+};0)$ with $\alpha_{\kappa^+}>\mu_1$, then we can always change the $\mu_1$ and $\nu_1$ strips into a long Seifert circle as shown in Figure \ref{MT1e2i_fig2} for the case of $\mu_1=5$. Where the thin curves indicate the re-routed under-strands from the $\mu_1$ strip and the thick curve indicate the re-routed over-strands from the $-\nu_1=-(\mu_1+1)$ strip. The boxes with shadows indicate where the $\alpha_j$ strips can be placed. Figure \ref{MT1e2i_fig2} is a demonstration of case from $P_3(5;-6\vert 16,14,12,12;0)$. Notice that this process can be carried out regardless the positions of the $\alpha_j$ strips.

\begin{figure}[htb!]
\includegraphics[scale=.8]{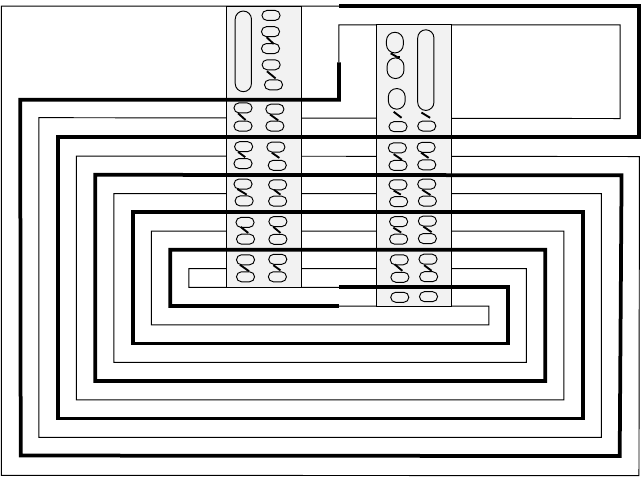}
\caption{The structure of the Seifert circle decomposition of a diagram $D$ of $\L\in P_3(5;-6\vert 16,14,12,12;0)$ after its $\mu_1$ and $\nu_1$ strips are combined into a long Seifert circle. In this case the strips in the diagram are ordered as follows:  $12$, $16$, $-6$, $14$, $12$, $5$. The over-strands that can be used for the special moves (or MP moves) are highlighted by a single line segment between the two small Seifert circles.} 
\label{MT1e2i_fig2} 
\end{figure}

\medskip
Case (ii) $\kappa^-=0,  \kappa^+>1, \mu_1=1, \nu_1=3, \alpha_{\kappa^+}=\alpha_{\kappa^+-1}=1$. $\L\in P_3(2;-3\vert 2\alpha_1,\ldots,2\alpha_{\kappa^+-2},2;0)$ in this case. A standard diagram $D$ with these parameters has 

$$s(D)=2+2\sum_{1\le j\le \kappa^+-1}\alpha_j-(\kappa^+-1)=1-\kappa^++2\sum \alpha_j$$
(keep in mind that $\alpha_{\kappa^+}=\alpha_{\kappa^+-1}=1$) with $-\kappa^++\sum \alpha_j$ MP-moves. Thus $D$ can be realized by a diagram with $1+\sum \alpha_j$ Seifert circles.

\medskip
Case (iii) $\kappa^-=0,  \kappa^+>0, \mu_1=1, \nu_1=2, \alpha_{\kappa^+}=1$. In this case $\L\in P_2(2\alpha_1,\ldots,2\alpha_{\kappa^+-1},2;-2)$ and the result follows from \cite[Theorem 1.4, formula (8)]{Diao2024}.

\medskip
Case (iv) $\kappa^-=0,  \kappa^+>0, \mu_1=1, \nu_1=3, \alpha_{\kappa^+}>1$. The argument here is similar to Case (v) below.

\medskip
Case (v) $\kappa^-=0,  \kappa^+>0, \mu_1=1, \nu_1\ge 4$. Let $D$ be a standard diagram of $\L$. In this case an N-move followed by a re-routing of the bottom strand of $D$ results in a diagram $\hat{D}$ such that $s(\hat{D})=s(D)-1=1-\kappa^++2\sum \alpha_j$, while keeping the number of local MP moves unchanged, that is, $r^+(D)=-\kappa^++\sum\alpha_j$. See Figure \ref{MT1e2v_fig} for an illustration of this process. It follows that $\L$ can be represented by a diagram with $1+\sum\alpha_j$ Seifert circles.

\begin{figure}[htb!]
\includegraphics[scale=.6]{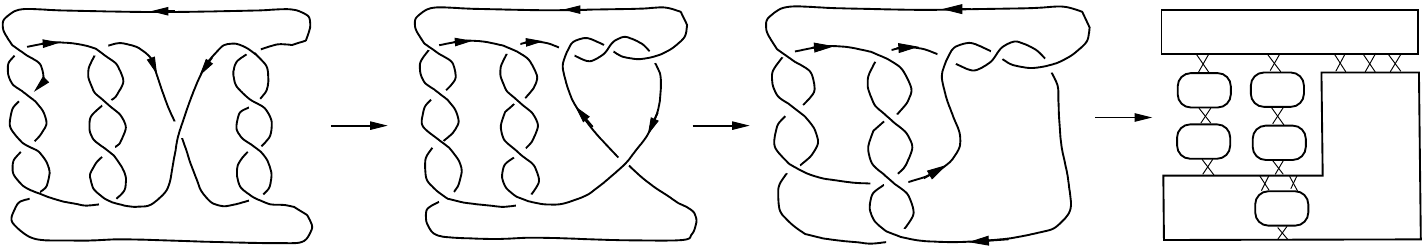}
\caption{The N-move followed by a re-routing of the bottom long strand reduces the number of Seifert circles by one, while keeping the number of local MP moves unchanged.} 
\label{MT1e2v_fig} 
\end{figure}

\medskip\noindent
Formula (\ref{MT1e3}): In all three cases here, we observe that if $D$ is a standard diagram of $\L$, then $s(D)=2-\kappa^++2\sum \alpha_j$ and $D$ has $-\kappa^++\sum \alpha_j$ MP-moves. Hence $\L$ can be represented by a diagram with $2+\sum\alpha_j$ Seifert circles.

\medskip\noindent
Formula (\ref{MT1e4}): In each case, the larger of the two numbers are given by the same argument for Formula (\ref{MT1e3}), while the smaller number is given by $\b_0(\L)$.

\medskip\noindent
Formula (\ref{MT1e5}): $\rho^+=\delta^+=2$, $\rho^-=0$, $\kappa^+\ge 2$. The given condition allows us to perform two A-moves on a standard diagram of $\L$, resulting in a diagram $\hat{D}$ with $s(\hat{D})=-\kappa^+-\kappa^-+2\sum \alpha_j+2\sum \beta_i$ and $-\kappa^+-\kappa^-+\sum \alpha_j+\sum \beta_i$ MP-moves, and the result follows. See Figure \ref{MT1e5_fig} for an example.

\begin{figure}[htb!]
\includegraphics[scale=.6]{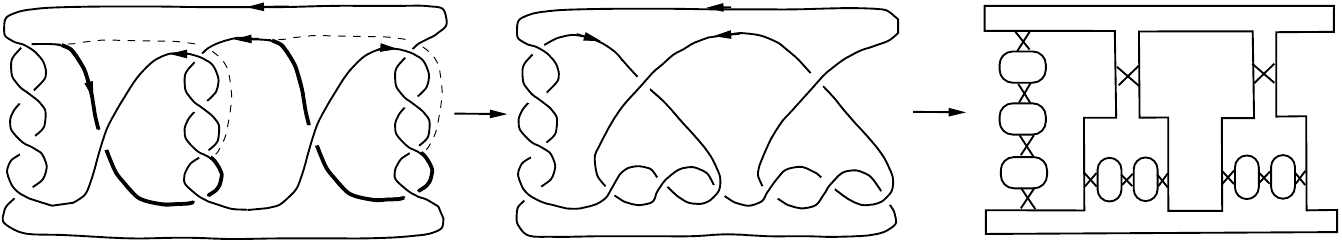}
\caption{Two A-moves on a diagram of $\L\in P_3(1,1;0\vert 4,4,4;0)$ result in a diagram with 9 Seifert circles which still allow $-\kappa^++\sum\alpha_j=3$ MP-moves. Thus $\L$ can be represented by a diagram with $\sum\alpha_j=6$ Seifert circles.} 
\label{MT1e5_fig} 
\end{figure}

\medskip\noindent
Formula (\ref{MT1e6}): 
Case (i) $\kappa^+>0,  \kappa^->0, \mu_1=1, \nu_1\ge 2$. Using flypes we can assume that  the $\mu_1$ and $\nu_1$ strips are adjacent. An N-move changes $D$ to a non alternating Type M2 diagram $\hat{D}$ as shown in Figure \ref{MT1e6_fig} in the middle. Since we can always find an adjacent pair of $\alpha_j$ and $\beta_i$ strips, the re-routing as shown in Figure \ref{MT1e6_fig} on the right is always possible. Notice that in the resulting diagram we have $1-\kappa^+-\kappa^-+2\sum\alpha_j+2\sum\beta_i$ Seifert circles, while the total number of MP moves remain unchanged, which is $-\kappa^+-\kappa^-+\sum\alpha_j+\sum\beta_i$. It follows that $\L$ can be represented by a diagram with $1+\sum\alpha_j+\sum\beta_i$ Seifert circles.

\begin{figure}[htb!]
\includegraphics[scale=.6]{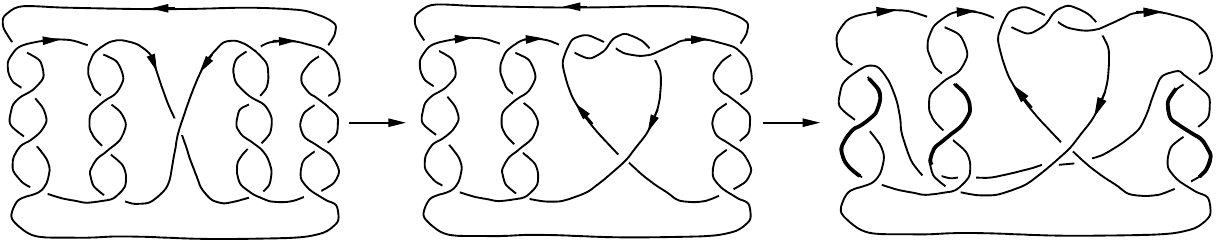}
\caption{The N-move following a re-routing of the top long strand reduces the number of Seifert circles by one, while keeping the number of local MP moves unchanged.} 
\label{MT1e6_fig} 
\end{figure}

Cases (ii) and (iii) $\rho^+=2$, $\rho^-=0$, $\delta^+=1$, $\kappa^+\ge 1$ or $\delta^+\ge 1$, $\kappa^+=1$. This is similar to the case of Formula (\ref{MT1e5}), but the given condition only allows us to perform one A-move, which gives us a diagram $\hat{D}$ with $s(\hat{D})=1-\kappa^+-\kappa^-+2\sum \alpha_j+2\sum \beta_i$ and $-\kappa^+-\kappa^-+\sum \alpha_j+\sum \beta_i$ MP-moves, and the result follows.

\medskip\noindent
Formula (\ref{MT1e7}): $2+\sum\alpha_j+\sum\beta_i$ is a braid index upper bound for all cases since a standard diagram always have $2-\kappa^+-\kappa^-+2\sum\alpha_j+2\sum\beta_i$ Seifert circles with $-\kappa^+-\kappa^-+\sum\alpha_j+\sum\beta_i$ MP-moves.

\medskip\noindent
Formula (\ref{MT2e1}): We only need to consider the case $\delta^+>\rho^-+\kappa^+$ as the other case is the mirror image of this. The given condition allows us to perform $\kappa^+$ A-moves (see Figure \ref{alpha_sigma_reduction}) without affecting the number of MP-moves. In the resulting diagram $\hat{D}$, we still have $\delta^+-\kappa^+$ lone crossings in the main cycle of Seifert circles (whose length of $2n$ remains unchanged). Thus we can still perform $\min\{\delta^+-\kappa^+,n-1\}$ MP-moves on these lone crossings. The A-moves have reduced the contribution of each $\alpha_j$ strip to the total number of Seifert circles by one, namely from $2\alpha_j-1$ to $2\alpha_j-2$, while we can still perform $\alpha_j-1$ MP-moves on the crossings in the remaining $\alpha_j$ strip. That is, we have $s(\hat{D})=2n-2\kappa^+-\kappa^-+2\sum\alpha_j+2\sum\beta_i$, and $\hat{D}$ has a total of $\min\{\delta^+-\kappa^+,n-1\}-\kappa^+-\kappa^-+\sum\alpha_j+\sum\beta_i$ MP-moves. Thus $\L$ can be represented by a diagram with $2n-\kappa^+-\min\{\delta^+-\kappa^+,n-1\}+\sum\alpha_j+\sum\beta_i$ Seifert circles.

\medskip\noindent
Formula (\ref{MT2e2}): If $\delta^+\le \kappa^+$, then we can perform $\delta^+$ A-moves and there are no more lone crossings left on the main cycle of Seifert circles in the resulting diagram, and the result follows from a similar calculation in the case of Formula (\ref{MT2e1}). If $\delta^+> \kappa^+$, then we can repeat the argument used for Formula (\ref{MT2e1}) to obtain the same expression. Since $\delta^+-\kappa^+<\rho^-$, we have $2(\delta^+-\kappa^+)<\delta^++\rho^--\kappa^+\le 2n-\kappa^+$ (since $\delta^+\le \rho^+$ and $\rho^++\rho^-=2n$), hence $\delta^+-\kappa^+\le n-1$ so $2n-\kappa^+-\min\{\delta^+-\kappa^+,n-1\}=2n-\delta^+$. This completes the proof of Theorem \ref{Type3UpboundTheorem}.
\end{proof}

\section{Further discussions}\label{end_sec}

We would like to spend this last section to further discuss the unsettled cases, namely the exceptional cases listed in Formula (\ref{MT1e4}). As we have shown in Examples \ref{Ex1}(iii), the link in $P_3(2;-3\vert 4;0)$ represented a case where $\b_0(\L)$ is strictly smaller than $\b(\L)$. By computing $\b_0$ for their corresponding parallel doubles, we are able to find a few more such examples, which we list below.

\medskip
\begin{examples}{\em 
Let $D$ be a standard diagram of $\L$, where $\L$ is one of the Type 3 pretzel links listed below, and let $\mathbb{D}$ be the parallel double of $D$, then we have $\b_0(\mathbb{D})=2\b_0(\L)+1$, it follows that $\b(\L)= 1+\b_0(\L)=2+\sum \alpha_j$ for each $\L$ in the list below.

\smallskip\noindent
(i) $\L\in P_3(2;-3\vert 4,4;0)$, $\b_0(\mathbb{D})=11$ and $\b(\L)=6$;

\smallskip\noindent
(ii) $\L\in P_3(2;-3\vert 4,4,4;0)$, $\b_0(\mathbb{D})=15$ and $\b(\L)=8$;

\smallskip\noindent
(iii) $\L\in P_3(2;-3\vert 4,6;0)$, $\b_0(\mathbb{D})=13$ and $\b(\L)=7$;

\smallskip\noindent
(iv) $\L\in P_3(4;-5\vert 4;0)$, $\b_0(\mathbb{D})=7$ and $\b(\L)=4$;

\smallskip\noindent
(v) $\L\in P_3(4;-5\vert 4,4;0)$, $\b_0(\mathbb{D})=11$ and $\b(\L)=6$.}
\end{examples}

\medskip
However, we should point out that this approach did not get us very far. For all examples that we have tried in which $\min\{\alpha_j\}\ge 3$, we have found that $\b_0(\mathbb{D})=2\b_0(\L)$ regardless of the $\mu_1$ value, which does not help us. On the other hand, if $D$ is a standard diagram and $\mathcal{D}$ is the parallel triple of $D$, a direct computation shows $\b(\L)=2+\sum\alpha_j$ for the list below:

\smallskip\noindent
(i) $\L\in P_3(3;-4\vert 6;0)$, $\b_0(\mathcal{D})=13$ and $\b(\L)=5$;

\smallskip\noindent
(ii) $\L\in P_3(4;-5\vert 6;0)$, $\b_0(\mathcal{D})=13$ and $\b(\L)=5$;

\smallskip\noindent
(iii) $\L\in P_3(3;-4\vert 6,6;0)$, $\b_0(\mathcal{D})=22$ and $\b(\L)=8$;

\medskip
Similar to the cases of the parallel doubles, for all examples in which we have used the parallel triples $\mathcal{D}$ with $\min\{\alpha_j\}\ge 4$, we have found that $\b_0(\mathcal{D})=3\b_0(\L)$ regardless of the $\mu_1$ value, which does not help us. One can speculate that for cases with $\min\{\alpha_j\}= 4$ a parallel quadruple might lead us to $\b(\L)=2+\sum \alpha_j$, however this approach become computationally prohibitive.  These computational results, together with how we prove Case (i) of Formula (\ref{MT1e2}), lead us to the following conjecture, which we use to end our comprehensive study of the braid indices of pretzel links.

\medskip
\begin{conjecture}{\em
If $\L\in P_3(\mu_1;-(\mu_1+1)\vert 2\alpha_1,\ldots,2\alpha_{\kappa^+};0)$ where $\mu_1>1$ and $1<\min\{\alpha_j\}\le \mu_1$, then 
 $\b(\L)=2+\sum\alpha_j$. If $\L\in P_3(1;-3\vert 2\alpha_1,\ldots,2\alpha_{\kappa^+-2},4,2;0)$ (with $\alpha_j\ge 2$, $1\le j\le \kappa^+-2$), then
 $\b(\L)=1+\sum\alpha_j$. Similarly, we conjecture that if $\L\in P_3(\nu_1+1;-\nu_1\vert 0;-2\beta_1,\ldots,-2\beta_{\kappa^-})$ where $\nu_1>1$ and $1<\min\{\beta_i\}\le \nu_1$, then 
 $\b(\L)=2+\sum\beta_i$, and if $\L\in P_3(3;-1\vert 0;-2\beta_1,\ldots,-2\beta_{\kappa^--2},-4,-2;0)$ (with $\beta_i\ge 2$, $1\le i\le \kappa^--2$), then
 $\b(\L)=1+\sum\beta_i$.}
\end{conjecture}

\end{document}